\documentclass[journal,twoside,web]{ieeecolor}

\usepackage{generic}
\usepackage[hidelinks]{hyperref}
\usepackage{url}
\usepackage{amsmath}
\usepackage{amssymb}
\usepackage{algorithm}  
\usepackage[style=ieee,backend=bibtex,url=false,doi=false,isbn=false,date=year,citestyle=numeric-comp,maxbibnames=10,minbibnames=10,maxcitenames=10,mincitenames=10]{biblatex}
\usepackage{graphicx}
\bibliography{paper.bib}
\usepackage{algpseudocode}
\usepackage{color}
\usepackage[normalem]{ulem}

\newtheorem{defi}{Definition}
\newtheorem{ass}{Assumption}
\newtheorem{thm}{Theorem}
\newtheorem{lem}{Lemma}

\newtheorem{rem}{Remark}
\newtheorem{example}{Example}

\newcommand{\Vf}{V_{\text{f}}}

\DeclareMathOperator*{\argmin}{argmin} 

\usepackage{caption}
\usepackage{subcaption}

\begin{document}
\title{Decentralized Real-Time Iterations for\\Distributed NMPC}
\author{Gösta Stomberg, \IEEEmembership{Member, IEEE}, Alexander Engelmann, \IEEEmembership{Member, IEEE}, Moritz Diehl, and \\Timm Faulwasser, \IEEEmembership{Senior Member, IEEE}
\thanks{This work was supported by the German Federal Ministry for Economic Affairs and Climate Action (BMWK) under agreement no. 03EI4043A (Redispatch3.0) and by the Deutsche Forschungsgemeinschaft (DFG, German Research Foundation) - project number 527447339.}
\thanks{GS, AE, and TF were with the Institute of Energy Systems, Energy Efficiency and Energy Economics, TU Dortmund University, 44227 Dortmund, Germany.}
\thanks{GS and TF are now with the Institute of Control Systems, Hamburg University of Technology, 21079 Hamburg, Germany (e-mail: goesta.stomberg@tu-dortmund.de, timm.faulwasser@ieee.org).}
\thanks{AE is now with logarithmo GmbH \& Co. KG, 44227 Dortmund, Germany (e-mail: alexander.engelmann@ieee.org).}
\thanks{MD is with the Department of Microsystems Engineering (IMTEK) and with the Department of Mathematics, University of Freiburg, 79110 Freiburg, Germany (e-mail: moritz.diehl@imtek.uni-freiburg.de).}
\thanks{© 2025 IEEE.  Personal use of this material is permitted.  Permission from IEEE must be obtained for all other uses, in any current or future media, including reprinting/republishing this material for advertising or promotional purposes, creating new collective works, for resale or redistribution to servers or lists, or reuse of any copyrighted component of this work in other works.}
\thanks{This article has been published in the IEEE Transactions on Automatic Control (Early Access). DOI: 10.1109/TAC.2025.3622000.}
}

\maketitle

\begin{abstract}
This article presents a Real-Time Iteration (RTI) scheme for distributed Nonlinear Model Predictive Control (NMPC).
The scheme transfers the well-known RTI approach, a key enabler for many industrial real-time NMPC implementations, to the setting of cooperative distributed control.
At each sampling instant, one outer iteration of a bi-level decentralized Sequential Quadratic Programming~(dSQP) method is applied to a centralized optimal control problem.
This ensures that real-time requirements are met and it facilitates cooperation between subsystems.
Combining novel dSQP convergence results with RTI stability guarantees, we prove local exponential stability under standard assumptions on the MPC design with and without terminal constraints.
The proposed scheme only requires neighbor-to-neighbor communication and avoids a central coordinator.
A numerical example with coupled inverted pendulums demonstrates the efficacy of the approach.
\end{abstract}


\allowdisplaybreaks
\section{Introduction}

Distributed control concerns the operation and control of cyber-physical systems, e.g., energy systems~\cite{Venkat2008} or robot formations~\cite{vanParys2017}.
Coupling in systems of systems can occur in the dynamics, via constraints, or through a common objective.
A key challenge for the design of optimization-based schemes for interconnected and cyber-physical systems is to reconcile cooperation with the computational burden, i.e., real-time feasibility is a must in applications.
On the far end of the spectrum, decentralized control schemes do not allow for cooperation but also do not require communication between subsystems~\cite{Siljak1991}.
The opposite is centralized MPC, where the network of subsystems is considered as one, large-scale, dynamical system which is controlled by a single controller.

Model Predictive Control (MPC), also known as receding-horizon optimal control, has seen tremendous industrial success catalyzed by the development of numerical schemes tailored to the dynamics and to the problem formulation.
Of particular importance for the implementation of centralized Nonlinear MPC (NMPC) are Real-Time Iteration (RTI) schemes~\cite{Diehl2002,Zavala2009,Wolf2016,Zanelli2021,Kapernick2014,Darup2019}.
Instead of solving the Optimal Control Problem (OCP) to full accuracy in each control step, an RTI scheme applies only one or a few iterations of an optimization method.

On the other hand, Distributed MPC (DMPC) decomposes the numerical optimization among the subsystems~\cite{Scattolini2009, Muller2017}.
DMPC design proceeds mainly along two dimensions: (i) the OCP formulation and (ii) the implementation of an optimization algorithm with the desired degree of decomposition or decentralization.

With respect to (i), one may design either an individual OCP for each subsystem~\cite{Dunbar2007,Muller2012,Varutti2012}, or a centralized OCP for the whole cyber-physical system~\cite{Conte2016,Bestler2019}.
Schemes with individual OCPs for the subsystems require few communication rounds per control step, enjoy small communication footprints, and allow for fast sampling rates~\cite{Dunbar2007,Muller2012,Varutti2012}. However, they only allow for limited cooperation.
Distributed approaches built upon centralized OCPs generally require multiple communication rounds per control step, but also allow for cooperation as the control tasks of all subsystems are encoded in the centralized OCP.
We therefore refer to the latter approach as cooperative DMPC in this article.

With respect to (ii), numerous optimization algorithms have been proposed to solve centralized OCPs online~\cite{Stewart2011,Summers2012,Hours2016,Stomberg2022}.
A distinction can be made between distributed and decentralized optimization methods.
Distributed optimization splits most computations between the subsystems and there exists a central entity which coordinates the subsystems~\cite{Bertsekas1989}. Decentralized optimization only requires neighbor-to-neighbor communication without a coordinator~\cite{Nedic2018}.

Ultimately, cooperative DMPC aims to combine the high performance of centralized MPC with the favorable communication structure of decentralized optimization.
Real-time requirements dictate that only a finite number of optimizer iterations can be executed in each control step, which limits performance and must be addressed in the stability analysis.
While RTI schemes have been pivotal in bringing NMPC to industrial applications, distributed counterparts for similar results are not available.
For linear systems, DMPC-specific OCP designs are presented in~\cite{Conte2016,Darivianakis2019} and stability under inexact optimization is analyzed in~\cite{Kohler2019,Giselsson2014}.
For nonlinear systems, cooperative DMPC schemes with stability guarantees are presented in~\cite{Stewart2011, Bestler2019}.
In both articles, a centralized OCP with a terminal penalty is formulated such that the OCP value function serves as a candidate Lyapunov function for the closed-loop system.
Stability is ensured in two different ways: Either a feasible-side convergent optimization method is employed and stability follows from standard arguments~\cite{Stewart2011}.
The drawback of this approach is that the optimization method requires a feasible initialization, which is difficult to implement in practice, and that each subsystem needs access to the dynamics of all subsystems in the network.
The second option for guaranteeing stability is to solve the OCP with the Alternating Direction Method of Multipliers (ADMM) until a tailored ADMM stopping criterion is met~\cite{Bestler2019}.
This approach presumes that ADMM converges linearly, i.e., sufficiently fast, to the OCP minimizer. 
However, such ADMM convergence guarantees for problems with non-convex constraint sets are, to the best of the authors' knowledge, yet unavailable.
Moreover, we note that the existing ADMM convergence guarantees for non-convex constraint sets~\cite{Themelis2020,Wang2019} so far do not allow for decentralized implementations.

Consequently, the existing stability guarantees of cooperative DMPC for coupled nonlinear systems either require feasible initialization and global model knowledge in \textit{all} subsystems or they rely on rather strong assumptions on the achieved optimizer convergence.
To overcome these limitations, we propose a novel decentralized RTI scheme for distributed NMPC.
Our scheme builds on a bi-level decentralized Sequential Quadratic Programming (dSQP) scheme~\cite{Stomberg2022a}.
On the outer level, the method uses an inequality-constrained SQP scheme, which leads to partially separable convex Quadratic Programs (QPs) to be solved in each SQP step.
On the inner level, these QPs are solved with ADMM which is guaranteed to converge and can be implemented in decentralized fashion, i.e., it does not require a central coordinator.
Our previous conference paper~\cite{Stomberg2022a} presents an earlier dSQP version including a stopping criterion for ADMM, but it does not consider real-time control applications and stability analysis.
The idea of executing only few iterations of a tailored optimization method to enable distributed NMPC in real-time is also used in~\cite{Hours2016}.
Therein, an augmented Lagrangian-based decomposition scheme is proposed for non-convex OCPs and suboptimality bounds for the optimizer solutions are derived, if changes in the system state between subsequent NMPC steps are small. 
While stability of the dynamical system is not formally discussed in~\cite{Hours2016}, the approach shares commonalities to the scheme developed in this article and we later give a more detailed comparison in Remark~\ref{rem:Hours} in Section~\ref{sec:drti}.

In the present paper, we explore the theoretical foundation of decentralized real-time iterations for DMPC via dSQP.
The real-time feasibility of our approach has successfully been validated in experiments with mobile robots and on embedded hardware~\cite{Stomberg2023,Stomberg2025a} and the computational scalability for large-scale systems is investigated in~\cite{Stomberg2025b}.
Specifically, this article presents two contributions: First, we derive novel dSQP convergence guarantees when the number of inner iterations is fixed instead of relying on an inexact Newton type stopping criterion as in~\cite{Stomberg2022a}.
Second, we combine the linear convergence of dSQP with the RTI stability guarantees from~\cite{Zanelli2021} to derive the local exponential stability of the system-optimizer dynamics in closed loop.
To the best of our knowledge, we are the first to study the system-optimizer convergence in DMPC \textit{and} provide the respective ADMM convergence guarantees.

The article is structured as follows: Section~\ref{sec:problem} states the control objective and presents the OCP.
Section~\ref{sec:rti} recalls RTI stability for centralized NMPC.
Section~\ref{sec:dsqp} explains the bi-level dSQP scheme and derives new q-linear convergence guarantees when the number of inner iterations per outer iteration is fixed.
Section~\ref{sec:drti} presents the stability of the distributed RTI scheme.
Section~\ref{sec:numer} analyzes numericals results for coupled inverted pendulums.

\textit{Notation:} 
Given a matrix $A$ and an integer $j$, $[A]_j$ denotes the $j$th row of $A$. For an index set $\mathcal{A}$, $[A]_\mathcal{A}$ denotes the matrix consisting of rows $ [A]_j$ for all $j \in \mathcal{A}$. Likewise, $[a]_j$ is the $j$th component of vector $a$ and $a_\mathcal{A}$ is the vector of components $[a]_j$ for all $j \in \mathcal{A}$. The concatenation of vectors $x$ and $y$ into a column vector is $(x,y)$.
Given scalars $a_1,\dots,a_n$, $A = \mathrm{diag}(a_1,\dots,a_n) \in \mathbb{R}^{n \times n}$ is the diagonal matrix where $[A]_{ii} = a_i$.
Likewise, given matrices $A_1,\dots,A_S$, $C = \mathrm{diag}(A_1,\dots,A_S)$ is a block diagonal matrix with block $[C]_{ii} = A_i$.
The Euclidean norm of a vector $a \in \mathbb{R}^{n}$ is denoted by $\|a \| \doteq \sqrt{a^\top a}$.
The spectral norm of a matrix $A \in \mathbb{R}^{n \times m}$ is denoted by $\| A \| \doteq \sigma_{\max}(A)$, the largest singular value of $A$.
The closed $\varepsilon$ neighborhood around a point $x^\star \in \mathbb{R}^n$ is denoted as $\mathcal{B}(x^\star,\varepsilon)$, i.e., $\mathcal{B}(x^\star,\varepsilon) \doteq \{x \in \mathbb{R}^{n} | \| x - x^\star \| \leq \varepsilon\}$.
We denote the natural numbers by $\mathbb{N}$, the natural numbers extended by zero by $\mathbb{N}_0 \doteq \{0\} \cup \mathbb{N}$, the set of integers by $\mathbb{I}$, the set of integers in the range from $0$ to $N$ by $\mathbb{I}_{[0,N]}$, and the Minkowski sum of sets $A$ and $B$ as $A \oplus B$.
Given $a \in \mathbb{R}$, $b = \lceil a \rceil$ is the nearest integer $b \geq a$.

\section{Problem Statement}\label{sec:problem}

Consider a network $\mathcal{S} = \{1, \dots, S\}$ of dynamical systems connected by a graph $\mathcal{G} = (\mathcal{S}, \mathcal{E})$, where the edges $\mathcal{E} \subseteq \mathcal{S} \times \mathcal{S}$ couple neighboring subsystems.
We define the set of subsystems which directly influence subsystem $i$ as in-neighbors $\mathcal{N}_i^\text{in} \doteq \{ j \in \mathcal{S} \setminus \{i\} \; | \; (j,i) \in \mathcal{E} \}$.
Similarly, we collect the subsystems which are influenced by subsystem $i$ in $\mathcal{N}_i^\text{out} \doteq \{ j \in \mathcal{S} \setminus \{i\} \; | \; (i,j) \in \mathcal{E}\}$.
We assume each subsystem $i \in \mathcal{S}$ can communicate with its neighbors $\mathcal{N}_i^\text{in}$ and $\mathcal{N}_i^\text{out}$.

We discuss distributed NMPC schemes for setpoint stabilization, where the subsystems cooperatively solve the OCP
\begin{subequations}\label{docp}
	\begin{align}
	\hspace*{-3mm}& \min_{ \bar{\boldsymbol{x}}, \bar{\boldsymbol{u}}}  \sum_{i \in \mathcal{S}} J_i(\bar{\boldsymbol{x}}_i,\bar{\boldsymbol{u}}_i,\bar{\boldsymbol{x}}_{\mathcal{N}_i^\text{in}})\label{docp:obj}\\
	\nonumber \hspace*{-3mm} \text{subject} & \text{ to } \text{for all  } i \in \mathcal S\\
	\hspace*{-3mm} \bar{x}_i(\tau \hspace*{-0.5mm}+\hspace*{-0.5mm}1) &= f_i^\delta(\bar{x}_i(\tau), \bar{u}_i(\tau), \bar{x}_{\mathcal{N}_i^\text{in}}(\tau)), \hspace*{5.5mm} \forall \tau \in \mathbb{I}_{[0,N-1]},\label{docp:sys}\\
	\hspace*{-3mm}\bar{x}_i(0) &= x_i(t),  \\
	\hspace*{-3mm}\bar{x}_i(\tau) &\in \mathbb{X}_i,  \; \forall \tau \in \mathbb{I}_{[0,N]}, \;\, \bar{u}_i(\tau) \in \mathbb{U}_i,  \; \forall \tau \in \mathbb{I}_{[0,N-1]},\\
	\hspace*{-3mm}(\bar{x}_i(\tau),&\bar{x}_j(\tau)) \in \mathbb{X}_{ij}, \hspace*{11.8mm} \forall j \in \mathcal{N}_i^\text{in},  \; \forall \tau \in \mathbb{I}_{[0,N]},\label{docp:xixj}
	\end{align}
\end{subequations}
with objective functions
\begin{equation*}
J_i (\cdot)  \doteq  \sum_{\tau = 0}^{N-1} \ell_i ( \bar{x}_i(\tau), \bar{u}_i(\tau), \bar{x}_{\mathcal{N}_i^\text{in}}(\tau) ) +  \beta V_{\text{f},i}(\bar{x}_i(N)).
\end{equation*}
The state and input of subsystem $i$ are denoted by $x_i \in \mathbb{R}^{n_{x_i}}$ and $u_i \in \mathbb{R}^{n_{u_i}}$, respectively. 
To distinguish closed-loop and open-loop trajectories, we denote predicted states and inputs with a superscript $\bar{\cdot}$.
The decision variables of OCP~\eqref{docp} are the predicted state trajectories $\bar{\boldsymbol{x}}_i$ and input trajectories $\bar{\boldsymbol{u}}_i$ over the horizon $N$.
Define $n_i^\text{in} \doteq \sum_{j \in \mathcal{N}_i^{\text{in}}} n_{x_j}$.
We stack the states of in-neighbors of subsystem $i$ in alphabetical order in the vector ${x}_{\mathcal{N}_i^\text{in}} \in \mathbb{R}^{n_i^\text{in}}$. 
The objective~\eqref{docp:obj} consists of individual stage costs $\ell_i: \mathbb{R}^{n_{x_i}} \times \mathbb{R}^{n_{u_i}} \times \mathbb R^{n_i^\text{in}} \rightarrow \mathbb{R}$, terminal penalties $V_{\text{f},i} : \mathbb{R}^{n_{x_i}} \rightarrow \mathbb{R}$, and a scaling factor $\beta \geq 1$.	 
For all $i \in \mathcal{S}$, $f_i^\delta : \mathbb{R}^{n_{x_i}} \times \mathbb{R}^{n_{u_i}} \times \mathbb{R}^{n_i^{\text{in}}} \rightarrow \mathbb{R}^{n_{x_i}}$ denotes the discrete-time dynamics with control sampling interval $\delta > 0$.
The states and inputs are constrained to the closed sets $\mathbb{X}_i \subseteq \mathbb{R}^{n_{x_i}}$, $\mathbb{U}_i \subseteq \mathbb{R}^{n_{u_i}}$.
Each closed set $\mathbb{X}_{ij} \subseteq \mathbb{R}^{n_{x_i}} \times \mathbb{R}^{n_{x_j}}$ couples two neighbors $i,j \in \mathcal{S}$.
We do not enforce additional terminal constraints to reduce the computational burden~\cite{Limon2006}.

In our stability analysis, we view OCP~\eqref{docp} for the network $\mathcal{S}$ as the following centralized OCP
\begin{subequations}\label{ocp}
	\begin{align}
	V&(x(t)) \doteq \min_{ \bar{\boldsymbol{x}}, \bar{\boldsymbol{u}}} \sum_{\tau = 0}^{N-1} \ell ( \bar{x}(\tau), \bar{u}(\tau) ) +  \beta V_{\text{f}}(\bar{x}(N)) \\
	\nonumber & \text{subject to} \\
	\bar{x}(\tau+1) &= f^\delta(\bar{x}(\tau), \bar{u}(\tau)), \hspace*{17.4mm} \forall \tau \in \mathbb{I}_{[0,N-1]},\label{ocp:sys}\\
	\bar{x}(0) &= x(t),\label{ocp:init} \\
	\bar{x}(\tau) &\in \mathbb{X}, \; \forall \tau \in \mathbb{I}_{[0,N]}, \;\; \bar{u}(\tau) \in \mathbb{U}, \; \forall \tau \in \mathbb{I}_{[0,N-1]}.
	\end{align}
\end{subequations} 
The centralized system state and input are $x = (x_1,\dots,x_S)\in \mathbb{R}^{n_x}$ and $u = (u_1,\dots,u_S) \in \mathbb{R}^{n_u}$, respectively.
The centralized discrete-time dynamics $f^\delta : \mathbb{R}^{n_{x}} \times \mathbb{R}^{n_{u}} \rightarrow \mathbb{R}^{n_{x}}$ are obtained by sampling the corresponding continuous-time dynamics $f^\mathrm{c}: \mathbb{R}^{n_{x}} \times \mathbb{R}^{n_{u}} \rightarrow \mathbb{R}^{n_{x}}$ with piecewise constant input signals at $\delta > 0$.
The partitioning of the centralized system into subsystems affects the performance of DMPC and is studied in~\cite{Chanfreut2021}.

The elements of OCP~\eqref{ocp} are comprised of the components of~\eqref{docp}.
Considering the entire network as a single system of high state dimension serves as a conceptual means in the stability analysis and allows us to draw upon existing RTI stability guarantees.
Notice that the numerical scheme to be proposed subsequently is decentralized, because we solve OCP~\eqref{docp} via dSQP~\cite{Stomberg2022a}.

\section{Centralized Real-Time Iterations}\label{sec:rti}

RTI schemes are designed to ensure the nominal closed-loop properties if only few optimizer iterations are taken in each control step to compute a control input for system~\eqref{ocp:sys}.
This section recalls stability guarantees that also hold when inequality constraints are present in the OCP~\cite{Zanelli2021}.

Let $\mathbb{X}_0 \subseteq \mathbb{R}^{n_x}$ be a closed set with nonempty interior and let OCP~\eqref{ocp} be feasible for all initial states $x(t) \in \mathbb{X}_0$. We define the NMPC control law $\kappa_\mathrm{c} : \mathbb{X}_0 \rightarrow \mathbb{R}^{n_u}$ as the map from the initial state $x(t)$ in~\eqref{ocp:init} to the first part $\bar u^{\star}(0)$ of a globally optimal input trajectory.
We then make the following assumption on the value function $V: \mathbb{X}_0 \rightarrow \mathbb{R}$, which can be met by appropriate OCP design, cf.~\cite{Limon2006}.
OCP designs specific to DMPC are discussed in~\cite{Dunbar2007,Conte2016,Darivianakis2019}.

\begin{ass}[Value function requirements~\cite{Zanelli2021}]\label{ass:Lyapunov} 
	~\\The value function $V: \mathbb{X}_0 \rightarrow \mathbb{R}$ of OCP~\eqref{ocp} is continuous and there exist positive constants $a_1$, $a_2$, $a_3$, and~$\bar{V}$ such that, for all $x \in \mathbb{X}_{\bar{V}} \doteq \{x \in \mathbb{R}^{n_x} \,|\, V(x) \leq \bar{V}\}$,
	\begin{subequations}\label{eq:lyapReq}
		\begin{align}
			a_1 \|x\|^2 \leq V(x) &\leq a_2 \| x \|^2\\
			V(f^\delta(x,\kappa_\mathrm{c}(x))) - V(x) &\leq - \delta \cdot a_3 \| x \|^2.
		\end{align}
	\end{subequations}
	Furthermore, there exists a constant $L_{V,x} > 0$ such that $| \sqrt{V(x)} - \sqrt{V(x')} | \leq L_{V,x}\| x - x' \|$ for all $x,x' \in \mathbb{X}_{\bar{V}}$. Moreover, there exists a constant $\hat{r}_x > 0$ such that $\mathbb{X}_{\bar{V}} \oplus \mathcal{B}(0,\hat{r}_x) \subseteq \mathbb{X}_0$. \hfill $\square$
\end{ass}
\begin{defi}\label{def:pbar}
	Let $\bar{p}: \mathbb{X}_0 \rightarrow \mathbb{R}^{n_p}$ be the map from the current state $x(t)$ in~\eqref{ocp:init} to a globally optimal primal-dual variable $p^\star$ of OCP~\eqref{ocp}. \hfill $\square$
\end{defi}
Note that the control $\bar u^\star(0)$ may be selected from the primal-dual variables $p^\star$ via a suitable matrix $M_{u,p} \in \mathbb{R}^{n_u} \times \mathbb{R}^{n_p}$ with $\|M_{u,p}\| = 1$, i.e., $\bar u^\star(0) = M_{u,p} \bar{p}(x)$.
\begin{ass}[Lipschitz controller~\cite{Zanelli2021}]\label{ass:Lipschitz_z}
	There exists a positive constant
	$L_{p,x}$ such that $\| \bar{p}(x') - \bar{p}(x) \| \leq L_{p,x} \| x' - x\|$ for all $x \in \mathbb{X}_{\bar{V}}$ and for all $x' \in \mathbb{X}_{\bar{V}} \cup \mathcal{B}(x,\hat{r}_x)$. Moreover, $\bar{p}(0) = 0$. \hfill $\square$
\end{ass}
Note that $\bar{p}(x')$ in Assumption~\ref{ass:Lipschitz_z} is well defined for all $x' \in \mathbb{X}_{\bar{V}} \cup \mathcal{B}(x,\hat{r}_x)$, because Assumption~\ref{ass:Lyapunov} ensures that OCP~\eqref{ocp} with initial state $x'$ is feasible. Further
note that $\bar{p}(0) = 0$ if the origin is an equilibrium of the system dynamics which minimizes $\ell$ and $V_\mathrm{f}$, a common design choice in stabilizing NMPC.
Put differently, this assumption does not hold for so-called economic NMPC schemes without further modification.

We continue by introducing the considered RTI scheme and the resulting system-optimizer dynmics.
Let the superscript $\cdot^k$ denote the iteration index of the optimization method which is used to solve OCP~\eqref{ocp}.
At time $t$, the state $x(t)$ is sampled and the optimization method is initialized with the solution from the previous control step, i.e., $p^0(t) = p^{k_\text{max}}(t-1)$, where $k_\text{max} \in \mathbb{N}$ is the number of optimizer iterations per control step.
Then, $k_\text{max}$ optimizer iterations are applied to OCP~\eqref{ocp}, a primal-dual iterate $p^{k_\text{max}}(t)$ is obtained, and the control input is selected, i.e. $u(t) = M_{u,p} p^{k_\mathrm{max}}(t)$.
The control input is applied to the system and the primal-dual iterates are stored for the next time step.
The optimization method and the system together form the system-optimizer dynamics~\cite{Zanelli2021}
\begin{equation}\label{eq:sysoptdyn}
\begin{bmatrix} x(t+1) \\ p^{k_\mathrm{max}}(t+1) \end{bmatrix} = \begin{bmatrix} f^\delta(x(t),M_{u,p}p^{k_{\mathrm{max}}}(t))  \\ \Phi(x(t),p^{k_\text{max}}(t)) \end{bmatrix},
\end{equation} where $\Phi: \mathbb{R}^{n_x} \times \mathbb{R}^{n_p} \rightarrow \mathbb{R}^{n_p}$ maps the OCP initial state $x(t)$ and approximate solution $p^{k_\text{max}}(t)$ to the approximate solution $p^{k_\text{max}}(t+1)$ at the next time step.
We tailor~\cite[Ass. 10]{Zanelli2021} by assuming optimizer convergence to the current OCP solution~$\bar{p}(x(t))$ as follows.
\begin{ass}[Q-linear optimizer convergence]\label{ass:qlin}
	There exist positive constants $\hat{r}_p>0$ and $a_p < 1$ such that, for all $x(t) \in \mathbb{X}_{\bar{V}}$ and all $p^0(t) \in \mathcal{B}(\bar{p}(x(t)),\hat{r}_p)$, the sequence $\{p^k(t)\}$ of optimizer iterates satisfies
	\begin{align*}
	\left\| p^{k+1}(t) - \bar{p}(x(t)) \right\| \leq a_p \left\| p^k(t) - \bar{p}(x(t)) \right\| \; \forall k \in \mathbb{N}_0.
	\end{align*}
\end{ass}
\vspace*{-0.4cm}
\hfill $\square$

\begin{ass}[Lipschitz system dynamics~\cite{Zanelli2021}]\label{ass:Lipschitz_sys}
	The centralized dynamics satisfy $f^\mathrm{c}(0,0) = 0$ and
	there exist positive finite constants $r_p'$, $\varrho$, $L_{f,x}^\mathrm{c}$, and $L_{f,u}^\mathrm{c}$ such that
	$$
	\| f^\mathrm{c}(x',u') - f^\mathrm{c}(x,u) \| \leq L_{f,x}^\mathrm{c} \| x' - x \| + L_{f,u}^\mathrm{c} \| u' - u \|
	$$ for all $x',x \in \mathbb{X}_{\bar{V}} \oplus \mathcal{B}(0,\varrho)$ and for all $u' = M_{u,p} p'$, $u = M_{u,p} p$ with $p',p \in \mathcal{B}(\bar{p}(x),r_p')$. \hfill $\square$	
\end{ass}

\begin{lem}[Centralized RTI stability]\label{lem:rti}
	Suppose that As-sumptions~\ref{ass:Lyapunov}--\ref{ass:Lipschitz_sys} hold and consider the sufficient sampling interval $\bar{\delta}$ and optimizer initialization radius $\tilde{r}_p$ defined in~\eqref{eq:RTIsuff} in Appendix~\ref{sec:app-rti}.	
	If $\delta \leq \bar \delta$, 
	then the origin is a locally exponentially stable equilibrium with region of attraction
	\begin{equation*}
	\Sigma \doteq \left\{ (x,p^{k_\mathrm{max}}) \in \mathbb{R}^{n_x + n_p} \left| \, x \in \mathbb{X}_{\bar{V}}, \left\| p^{k_\mathrm{max}} - \bar{p}(x) \right\| \leq \tilde r_p \right. \right\}
	\end{equation*}	
	for the closed-loop system-optimizer dynamics~\eqref{eq:sysoptdyn}. \hfill $\square$
\end{lem}
\begin{proof}
	The proof proceeds similarly to the analysis in~\cite{Zanelli2021}.
	Since our q-linear optimizer convergence Assumption~\ref{ass:qlin} differs from~\cite[Ass. 10]{Zanelli2021}, we replace~\cite[Lem. 11]{Zanelli2021} by Lemma~\ref{lem:optimizer-contraction} from the Appendix.
	Specifically, the optimizer contraction inequality~\eqref{eq:optimizer-contraction} derived in Lemma~\ref{lem:optimizer-contraction} is equivalent to~\cite[Ineq. 26]{Zanelli2021} in the proof of~\cite[Prop. 16]{Zanelli2021}.
	The remaining analysis in~\cite{Zanelli2021} stays applicable such that local exponential stability follows from~\cite[Thm. 25]{Zanelli2021}.
\end{proof}

\begin{rem}[Relation to~\cite{Zanelli2021}]
	Lemma~\ref{lem:rti} is a specialized variant of~\cite[Thm. 25]{Zanelli2021} in four aspects: 
	First, we only consider quadratic bounds on the Lyapunov function.
	Second, we define the map $\bar{p}$ from the current state to the primal-dual variables, because of the dSQP convergence presented in the next section.
	Third, we allow for $k_\mathrm{max} \geq 1$ optimizer iterations.
	In contrast,~\cite[Thm. 25]{Zanelli2021} considers more general Lyapunov functions, allows for any Lipschitz continuous map $\bar{p}: \mathbb{R}^{n_x} \rightarrow \mathbb{R}^{n_p}$ such that $\bar{u}^\star(0) = M_{u,p}\bar{p}(x)$, and considers $k_\mathrm{max} = 1$. Note that the extension to $k_\mathrm{max} > 1$ is straight forward because of the q-linear optimizer convergence in Assumption~\ref{ass:qlin}.
	Moreover, the rationale of this section is to first select a candidate sampling interval $\delta$ such that Assumptions~\ref{ass:Lyapunov}--\ref{ass:Lipschitz_sys} hold and to then infer stability if $\delta \leq \bar{\delta}$. 
	The approach of~\cite[Thm. 25]{Zanelli2021} instead lets an assumption similar to Assumption~\ref{ass:Lyapunov} hold for a range of sampling intervals and then guarantees stability \textit{for all} sampling intervals below $\bar{\delta}$. 
	\hfill $\square$
\end{rem}

\section{Decentralized Sequential Quadratic Programming}\label{sec:dsqp}

The crucial requirement on the optimization method to guarantee closed-loop stability via Lemma~\ref{lem:rti} is local q-linear convergence.
Hence we now recall dSQP from~\cite{Stomberg2022a} and derive the required convergence property.

By introducing state trajectory copies for neighboring subsystems, OCP~\eqref{docp} can be written as a partially separable Nonlinear Program (NLP)~\cite{Summers2012,Bestler2019}
\begin{subequations} \label{eq:sepForm}
	\begin{align} 
	\min_{z} \; \sum_{i \in \mathcal{S}} & \; f_i(z_i)\label{eq:sepProbFi} \\
	\text{subject to}\hspace{0.5mm} \quad  g_i(z_i)&=0 \; | \; \nu_i \quad \forall i \in \mathcal{S}, \label{eq:sepProbGi} \\
	h_i(z_i) &\leq 0 \; | \; \mu_i \hspace*{0.3cm} \forall i \in \mathcal{S},\label{eq:sepProbHi} \\
	\sum_{i \in \mathcal{S}} E_iz_i  &= c \hspace*{1.2mm} | \; \lambda.\label{eq:consConstr}
	\end{align}
\end{subequations} 
The decision variables $z_i \hspace*{-0.1cm}\in \hspace*{-0.1cm}\mathbb{R}^{n_i}$ of subsystem $i$ include the predicted trajectories $\bar{\boldsymbol{x}}_i$ and $\bar{\boldsymbol{u}}_i$ over the horizon as well as copies of the predicted state trajectories of neighboring subsystems.
Example~\ref{ex:coupling} in Appendix~\ref{app:example} illustrates the reformulation of OCP~\eqref{docp} as a partially separable NLP like~\eqref{eq:sepForm}.
The functions $f_i: \mathbb{R}^{n_i} \rightarrow \mathbb{R}$, $g_i:\mathbb{R}^{n_i} \rightarrow \mathbb{R}^{n_{g_i}}$, and ${h_i:\mathbb{R}^{n_i} \rightarrow  \mathbb{R}^{n_{h_i}}}$ are composed of the objective functions, equality constraints, and inequality constraints in OCP~\eqref{docp}, respectively.
\begin{ass}[Differentiability of the NLP functions]\label{ass:c3}
The functions $f_i$, $g_i$, and $h_i$ are three times continuously differentiable for all $i \in \mathcal{S}$. \hfill $\square$
\end{ass}
The sparse matrices $E_i \in \mathbb{R}^{{n_c} \times {n_i} }$ couple the subsystems by matching original and copied variables of state trajectories.
Thus, the sparsity pattern of the matrices $E_i$ arises from the coupling graph $\mathcal{G}$.
Specifically, $(i,j)$ or $(j,i) \in \mathcal{E}$ if and only if there exists an index $o \in \{1,\dots,n_c\}$ such that $[E_i]_o \neq 0$ and $[E_j]_o \neq 0$.

The notation in NLP~\eqref{eq:sepForm} highlights that $\nu_i \in \mathbb{R}^{n_{g_i}}$, $\mu_i \in \mathbb{R}^{n_{h_i}}$, and $\lambda \in \mathbb{R}^{n_c}$ are Lagrange multipliers associated with the respective constraints.
The centralized variables are $z\doteq (z_1,\dots,z_S) \in \mathbb{R}^n$, $\nu\doteq (\nu_1,\dots,\nu_S) \in \mathbb{R}^{n_g}$, and $\mu\doteq (\mu_1,\dots,\mu_S) \in \mathbb{R}^{n_h}$.
Likewise, we denote the centralized constraints as $g(z)\doteq (g_1(z_1),\dots,g_S(z_S))$ and $h(z)\doteq (h_1(z_1),\dots,h_S(z_S))$.
Throughout this section, we denote the globally optimal Karush-Kuhn-Tucker~(KKT) point of NLP~\eqref{eq:sepForm} with initial condition $x \in \mathbb{X}_{\bar{V}}$ as $p^\star = (z^\star,\nu^\star,\mu^\star,\lambda^\star) \doteq \bar{p}(x)$, cf.~Definition~\ref{def:pbar}.
That is, the notation $p^\star$ drops the explicit dependence on $x$ for simplicity and we keep in mind that the KKT point depends on the initial state.

We define the Lagrangian to NLP~\eqref{eq:sepForm} as
\begin{align*}
L(z,\nu,\mu,\lambda) &= \sum_{i \in \mathcal{S}} L_i(z_i,\nu_i,\mu_i,\lambda) - \lambda^\top c,
\end{align*}
where $L_i \doteq \sum_{i \in \mathcal{S}} (f_i(z_i) + \nu_i^\top g_i(z_i) + \mu_i^\top h_i(z_i) + \lambda^\top E_i z_i)$.

The bi-level dSQP method from~\cite{Stomberg2022a} combines an SQP scheme on the outer level with ADMM on the inner level. We index outer iterations by $\cdot^k$ and inner iterations by $\cdot^l$.
Starting from a primal-dual point $p^k = (z^k,\nu^k,\mu^k,\lambda^k)$, the method proceeds as follows.
In each SQP iteration, we first construct a quadratic approximation of NLP~\eqref{eq:sepForm}
\begin{subequations}\label{eq:QP}
	\begin{align}
	\min_{z} \sum_{i \in \mathcal{S}} f_i^{\mathrm{QP},k}(&z_i)  \label{eq:SQPobj}\\
	\textrm{subject to} \quad g_i^k + \nabla g_i^{k\top} (z_i-z_i^k) &= 0 \;| \; \nu_i \hspace{3.5mm} \forall i \in \mathcal S,\label{eq:QPeq}\\
	h_i^k + \nabla h_i^{k\top} (z_i-z_i^k) &\leq 0 \; | \; \mu_i \hspace{3mm} \forall i \in \mathcal S,\\
	\sum_{i \in \mathcal{S}} E_i z_i &= c \hspace*{0.1mm}\; | \; \lambda, 
	\end{align}
\end{subequations} 
where, $f_i^{\mathrm{QP},k} \doteq (z_i-z_i^k)^{\top} H_i^k (z_i-z_i^k)/2 + \nabla f_i^{k\top}(z_i-z_i^k) $, 
and where, for all $i \in \mathcal S$, $H_i^k \approx \nabla_{z_iz_i}^2 L_i(z_i^k,\nu_i^k,\mu_i^k)$ is positive definite on the space spanned by~\eqref{eq:QPeq}.
The symbols $g_i^k$ and $\nabla g_i^k$ in QP~\eqref{eq:QP} are shorthands for $g_i(z_i^k)$ and $\nabla g_i(z_i^k)$, respectively. The same holds for functions $f_i$ and $h_i$.
We denote the unique centralized primal dual-solution to QP~\eqref{eq:QP} as~$p^{k,\star}_{\mathrm{QP}} = (z^{k,\star}_{\mathrm{QP}},\nu^{k,\star}_{\mathrm{QP}},\mu^{k,\star}_{\mathrm{QP}},\lambda^{k,\star}_{\mathrm{QP}})$ and the centralized Hessian as $H^k \doteq \mathrm{diag}(H_1^k,\dots,H_S^k)$.

Then, we apply a fixed number $l_\text{max}$ of ADMM iterations to QP~\eqref{eq:QP}. To this end, we introduce the decision variable $y \in \mathbb{R}^{n}$ and reformulate QP~\eqref{eq:QP} in two-block form as
\begin{subequations}\label{eq:QPadmm} 
	\begin{align}
	\min_{\substack{y_1 \in \mathbb{Z}_1^k, \dots, y_S \in \mathbb{Z}_S^k\\ z\in \mathbb{E}} }  & \quad \sum_{i \in \mathcal{S}} f_i^{\mathrm{QP},k}(y_i)\label{eq:QPadmmObj}\\
	\textrm{subject to } \quad  y_i - z_i &= 0 \; | \; \gamma_i \quad \quad \forall i \in \mathcal{S}\label{eq:QPadmmCons}.	
	\end{align}
\end{subequations} 
The subsystem constraint sets are defined as
\begin{align*}
\mathbb{Z}_i^k \doteq \left\{ y_i \in \mathbb{R}^{n_i} \left| \begin{aligned} g_i^k + \nabla g_i^{k\top} (y_i-z_i^k)&= 0\\
h_i^k + \nabla h_i^{k\top} (y_i-z_i^k) &\leq 0 \end{aligned} \right. \right\},
\end{align*} the consensus constraint set is
\begin{align*}
\mathbb{E} \doteq \left\{ z \in \mathbb{R}^n \left| \sum_{i \in \mathcal{S}} E_i z_i = c \right. \right\},
\end{align*} and $\gamma_i \in \mathbb{R}^{n_i}$ is the Lagrange multiplier to constraint~\eqref{eq:QPadmmCons}.
We denote the centralized multiplier as $\gamma \doteq (\gamma_1,\dots,\gamma_S)$ and define the augmented Lagrangian for QP~\eqref{eq:QPadmm} as
\begin{align*}
L^k_\rho(y,z,\gamma) &= \sum_{i \in \mathcal{S}} L_{\rho,i}^k(y_i,z_i,\gamma_i),
\end{align*} where $L_{\rho,i}^k \doteq  f_i^{\mathrm{QP},k}(y_i) + \gamma_i^\top (y_i-z_i) + \rho \| y_i - z_i \|_2^2/2$ and where $\rho > 0$ is a penalty parameter.
The ADMM iterations in the centralized variables read~\cite{Boyd2011}
\begin{subequations}
	\begin{align}
	(y^{l+1},\nu^{l+1},\mu^{l+1}) &\leftarrow \min_{y \in \mathbb{Z}^k} L_{\rho}^k(y,z^l,\gamma^l)\label{eq:admm1}\\
	({z}^{l+1}, \lambda^{l+1} ) &\leftarrow \min_{z \in \mathbb{E}} L_{\rho}^k(y^{l+1},z,\gamma^l)\label{eq:admm2}\\
	\gamma^{l+1} &= \gamma^{l} + \rho  (y^{l+1} - {z}^{l+1})\label{eq:admm3}
	\end{align}
\end{subequations}
where the notation in \eqref{eq:admm1}--\eqref{eq:admm2} indicates that we update the primal and dual iterates with the primal-dual solution obtained in the respective step and where $\mathbb{Z}^k \doteq \mathbb{Z}_1^k \times \dots \times \mathbb{Z}_S^k$.
ADMM is given in Algorithm~\ref{alg:admm}.

The solution returned by ADMM is then used in the next SQP iteration to construct a new QP. The resulting dSQP method is summarized in Algorithm~\ref{alg:d-SQP}.

\begin{algorithm}[t]
	\caption{ADMM for solving QP~\eqref{eq:QP}~\cite{Boyd2011}}
	\begin{algorithmic}[1]
		\State Initialization: $z_i^0,\gamma_i^0$, $l_\text{max}$ for all $i \in \mathcal{S}$, \label{admm-stp:1}
		\For{ $l = 0, 1, \dots, l_\text{max} - 1$}
		\State \hspace*{-3mm}$\displaystyle \left(y_i^{l+1},\nu_i^{l+1},\mu_i^{l+1}\right) \leftarrow \min_{y_i \in \mathbb{Z}_i^k}  L_{\rho,i}^k\left(y_i,z_i^l,\gamma_i^l\right)$ for all $i \in \mathcal{S}$\label{admm-step:1}   
		\State \hspace*{-3mm}$\displaystyle z^{l+1} = \argmin_{z \in \mathbb{E}} L_{\rho}^k(y^{l+1},{z},\gamma^{l})$ \label{admm-step2}
		\State \hspace*{-3mm}$\gamma_i^{l+1} = \gamma_i^l + \rho \left(y_i^{l+1}-{z}_i^{l+1}\right)$ for all $i \in \mathcal{S}$\label{admm-step3}
		\EndFor
		\State \textbf{return} ${z}_i^{l_\text{max}},\nu_i^{l_\text{max}},\mu_i^{l_\text{max}},\gamma_i^{l_\text{max}}$ for all $i \in \mathcal{S}$\label{admm:return}	
	\end{algorithmic} \label{alg:admm}	
\end{algorithm}

\begin{algorithm}[t]
	\caption{dSQP for solving NLP~\eqref{eq:sepForm}~\cite{Stomberg2022a}}
	\begin{algorithmic}[1]
		\State Initialization: $z_i^0,\nu_i^0,\mu_i^0,\gamma_i^0 = E_i^\top \lambda^0 \; \forall \; i \in \mathcal{S}$, $k_\text{max},l_\text{max}$ \label{dsqp-stp:1}
		\For{$k = 0,1,\dots, k_\text{max}-1$ } \label{dsqp-step:2}
		\State \hspace*{-3mm}evaluate $\nabla f_i^k, g_i^k, \nabla g_i^k, h_i^k, \nabla h_i^k, H_i^k$ for all $i \in \mathcal{S}$
		
		\hspace*{-4.8mm}and build QP~\eqref{eq:QP} \label{stp:3}
		\State \hspace*{-3mm}initialize Algorithm~\ref{alg:admm} with $z_i^k,\gamma_i^k,l_\text{max}$ and denote the

		\hspace*{-4.8mm}output by $z_i^{k+1}, \nu_i^{k+1}, \mu_i^{k+1}, \gamma_i^{k+1}$ for all $i \in \mathcal{S}$
		\EndFor\label{euclidendwhile}
		\State \textbf{return} $z_i^{k_\text{max}}, \nu_i^{k_\text{max}}, \mu_i^{k_\text{max}}, \gamma_i^{k_\text{max}}$ for all $i \in \mathcal{S}$		
	\end{algorithmic} \label{alg:d-SQP}	
\end{algorithm}

\begin{rem}[Decentralized ADMM]\label{rem:averaging}
	Steps~\ref{admm-step:1} and~\ref{admm-step3} of Algorithm~\ref{alg:admm} can be carried out by each subsystem individually.
	For the considered OCPs, each constraint in~\eqref{eq:consConstr} couples exactly two subsystems to match original and copied state variables.
	Moreover, $c = 0$. 
	Note that every NLP can be transformed to have these so-called 2-assigned constraints by adding further decision variables~\cite{Engelmann2020c}.
	Thus, NLP~\eqref{eq:sepForm} is a consensus optimization problem and ADMM can be decentralized~\cite[Ch. 7]{Boyd2011}.
	Appendix~\ref{app:example} illustrates the decentralization of~\eqref{eq:admm2} on a small-scale example.
	Following~\cite{Nedic2018}, we therefore refer to the employed ADMM variant and in consequence to dSQP as \textit{decentralized} optimization methods. \hfill $\square$
\end{rem}
\begin{rem}[Communication requirements]
The only step of the proposed dRTI scheme which requires communication is Step~\ref{admm-step2} of ADMM.
As discussed in Remark~\ref{rem:averaging}, this step can be implemented as an efficient averaging step which only requires neighbor-to-neighbor communication.
A common procedure considered in DMPC is to implement the averaging step via two communication rounds: First, copied state trajectories are sent to in-neighbors. Then, averaged state trajectories are sent to out-neighbors, cf.~\cite{Bestler2019,Stomberg2023,Stomberg2025a}. \hfill $\square$
\end{rem}

In the following, we combine convergence results from inexact SQP schemes for the outer level and from ADMM for the inner level to pave the road towards decentralized RTI schemes with stability guarantees via novel q-linear convergence guarantees for~Algorithm~\ref{alg:d-SQP}.

\subsection{Outer Convergence}

We first consider a basic SQP method, where QP~\eqref{eq:QP} is solved to high accuracy in each SQP iteration~\cite{Boggs1995}.
Then, we move on to inexact SQP schemes which use approximate solutions of QP~\eqref{eq:QP}.

We define convergence to $p^{\star} = (z^{\star},\nu^{\star},\mu^{\star},\lambda^{\star})$ as follows.

\begin{defi}[Convergence rates]
	The sequence $\{p^k\} \subset \mathbb{R}^{n_p}$ is said to converge to $p^\star \in \mathbb{R}^{n_p}$
	\begin{enumerate}
		\item[(i)] q-linearly, if $\| p^{k+1} - p^\star \| \leq c \| p^k - p^\star \|$ for all $k \geq k_0$, for some $0<c<1$, and for some $k_0 \in \mathbb{N}_0$.
		\item[(ii)] q-quadratically, if ${p^k \rightarrow p^\star}$ and if there exists a $C > 0$ such that $\|p^{k+1} - p^\star \| \leq C \|p^k - p^\star \|^2$ for all $k \in \mathbb{N}_0$. \phantom{a}\hfill $\square$
	\end{enumerate} 
\end{defi}
For all $i \in \mathcal{S}$, denote the set of active inequality constraints at $z_i^\star$ as $\mathcal{A}_i \doteq \{ j \in \{1,\dots,n_{h_i}\} \; | \; [h_i(z_i^\star)]_j = 0\}$.
Likewise, we denote the set of inactive inequality constraints as $\mathcal{I}_i \doteq \{j \in \{1,\dots,n_{h_i}\} \; | \; [h_i(z_i^\star)]_j < 0\}$.
Recall that we assume the functions in NLP~\eqref{eq:sepForm} to be three times continuously differentiable.
An exact SQP scheme is locally convergent under the following assumption~\cite{Boggs1995}.

\begin{ass}[Regularity of $p^\star$]\label{ass:kkt}
	The point $p^\star$ is a KKT point of NLP~\eqref{eq:sepForm} which, for all $i \in \mathcal{S}$, satisfies
	\begin{enumerate}
		\item[(i)] $h_i(z_i^\star) + \mu_i^\star \neq 0$ (strict complementarity), 
		\item[(ii)] $z_i^\top \nabla_{z_iz_i}^2 L_i(z_i^\star,\nu_i^\star,\mu_i^\star) z_i > 0$ for all $z_i \neq 0$ with $\nabla g_i(z_i^\star)^\top z_i = 0$.
	\end{enumerate}
	Furthermore, 
	the matrix
	\begin{align*}
	\begin{bmatrix} \nabla_{z_1} g_1(z_1^\star)^\top & & \\
	& \ddots & \\
	& & \nabla_{z_S} g_S(z_S^\star)^\top\\
	[\nabla_{z_1} h_1(z_1^\star)^\top]_{\mathcal{A}_1} & & \\
	& \ddots & \\
	& &[\nabla_{z_S} h_S(z_S^\star)^\top]_{\mathcal{A}_S}\\
	E_1 & \dots & E_S \end{bmatrix}
	\end{align*} has full row rank, i.e., 
	$z^\star$ satisfies the Linear Independence Constraint Qualification (LICQ).\hfill $\square$ 
\end{ass}

\begin{lem}[{Exact SQP convergence \cite{Boggs1995,Nocedal2006}}]\label{lem:sqpConv}
	~\\Suppose Assumption~\ref{ass:c3} holds and let $p^\star$ denote a KKT point which satisfies Assumption~\ref{ass:kkt}.
	Consider an exact-Hessian SQP scheme with iteration $p^{k+1} = p^{k,\star}_\mathrm{QP}$, where QP~\eqref{eq:QP} is formed with $H^k = \nabla_{zz}^2 L(z^k,\nu^k,\mu^k)$.
	Then, there exists an $\varepsilon_1 > 0$ such that, for all $p^0 \in \mathcal{B}( p^\star,\varepsilon_1)$, the sequence $\{p^k\}$ generated by the exact SQP scheme converges q-quadratically to $p^\star$. \hfill $\square$
\end{lem}
Under Assumptions~\ref{ass:c3} and \ref{ass:kkt}, the SQP scheme considered in Lemma~\ref{lem:sqpConv} identifies the correct active set if $p^k \approx p^\star$~\cite[Sec. 5]{Boggs1995}. The above convergence result thus follows from~\cite[Thm. 18.4]{Nocedal2006}.
Note that Assumption~\ref{ass:kkt}~(ii) is slightly stronger than the standard Second-Order Sufficient Condition~(SOSC)~\cite[A4]{Boggs1995}, because  we exclude the conditions $[{\nabla h_i(z_i^\star)^\top]_{\mathcal{A}_i} z_i  = 0}$ and $E_iz_i=0$ on $z_i$ when demanding positive definiteness of the Hessian.
This facilitates the ADMM convergence analysis presented below and ensures that QP~\eqref{eq:QP} has a unique KKT point if $p^k \approx p^\star$.
Note that the latter does not necessarily hold under the standard SOSC condition, which can be addressed in the convergence analysis of SQP schemes by selecting $p^{k+1}$ as the KKT point of QP~\eqref{eq:QP} which is closest to $p^k$, cf.~\cite{Robinson1974}.

The exact SQP scheme considered in Lemma~\ref{lem:sqpConv} serves as the prototype for dSQP.
However, the real-time requirements in control only allow for a small number $l_\text{max}$ of ADMM iterations per SQP iteration in Algoritm~\ref{alg:d-SQP}.
Hence, we now consider an inexact SQP scheme where $p^{k+1}$ only approximates the primal-dual solution $p^{k,\star}_{\mathrm{QP}}$ of QP~\eqref{eq:QP}.

\begin{lem}[{Inexact SQP convergence}]\label{lem:truncSQP}
	Suppose Assumption~\ref{ass:c3} holds and let $p^\star$ denote a KKT point which satisfies Assumption~\ref{ass:kkt}.
	Form QP~\eqref{eq:QP} at a primal-dual point $p^k = (z^k,\nu^k,\mu^k,\lambda^k)$ using the exact Hessian $H^k = \nabla_{zz}^2 L(z^k,\nu^k,\mu^k)$. 
	Consider an inexact SQP scheme, whose iterates $\{p^{k}\}$, for all $k \in \mathbb{N}_0$ and for some $a < 1$, satisfy 
	\begin{equation}\label{eq:A1}
	\| p^{k+1} - p^{k,\star}_{\mathrm{QP}} \| \leq a \|p^k - p^{k,\star}_{\mathrm{QP}} \|.
	\end{equation}
	Furthermore, let $\bar{a}_p \in (a,1)$.		
	Then, there exists $\varepsilon_2 > 0$ such that the following holds.
	If $p^0 \in \mathcal{B} ( p^\star, \varepsilon_2)$, then the sequence $\{p^k\}$ generated by the inexact SQP scheme converges q-linearly to $p^\star$, $\|p^{k+1} - p^\star \| \leq \bar{a}_p \|p^k - p^\star \|$ for all $k \in \mathbb{N}_0.$~\hfill $\square$
\end{lem}
\begin{proof}
	By Lemma~\ref{lem:sqpConv}, the convergence radius of the exact SQP scheme is thus given by $\varepsilon_1 > 0$.
	Let $\varepsilon_2 \leq \varepsilon_1$. From Lemma~\ref{lem:sqpConv} we have for all $k \in \mathbb{N}_0$ that QP~\eqref{eq:QP} is feasible and
	\begin{equation}\label{eq:quadr}
	\| p^{k,\star}_{\mathrm{QP}} - p^\star \| \leq C \| p^k - p^\star\|^2.
	\end{equation}
	Combining \eqref{eq:A1} and \eqref{eq:quadr} yields
	\begin{align*}
	\| p^{k+1} - p^\star \| & \leq \| p^{k+1} - p_{\mathrm{QP}}^{k,\star} \| + \| p^{k,\star}_{\mathrm{QP}} - p^\star \|\\
	& \leq a \|p^k - p^{k,\star}_{\mathrm{QP}} \| + \| p^{k,\star}_{\mathrm{QP}} - p^\star \|\\
	& \leq a \|p^k - p^{\star}\| + (1+a) \| p^{k,\star}_{\mathrm{QP}} - p^\star \|\\
	& \leq a \|p^k - p^{\star}\| + (1+a) C \|p^k - p^{\star}\|^2.
	\end{align*}
	Choosing $\|p^0 - p^\star\| < (\bar{a}_p-a)/((1+a)C)$ with $a < \bar{a}_p < 1$, we obtain for all $k \in \mathbb{N}_0$
	\begin{equation*}
	\| p^{k+1} - p^\star \|  \leq \bar{a}_p \| p^{k} - p^\star \|. 
	\end{equation*}
	Hence, if $p^0 \in \mathcal{B}(p^\star,\varepsilon_2)$ and if $\varepsilon_2 \leq \text{min}(\varepsilon_1, (\bar{a}_p-a)/((1+a)C))$, then the sequence $\{p^k\}$ converges q-linearly to $p^\star$.	 
\end{proof}

Lemma~\ref{lem:truncSQP} guarantees local SQP convergence, if the inexact QP solutions satisfy inequality~\eqref{eq:A1} as in~\cite[Lem. 3.1.10]{Zanelli2021b}.
As we show next, the solutions produced by ADMM meet this requirement. 

\subsection{Inner Convergence}

We next derive a sufficiently large number of ADMM iterations $l_\mathrm{max}$ that guarantees inexact SQP convergence via~\eqref{eq:A1}.
To this end, we first express $l_\mathrm{max}$ based on constants associated with the ADMM convergence for convex QPs.
Subsequently, we quantify the constants to compute $l_\mathrm{max}$ inside a local convergence region where the active set stays constant.
Define $c_1 \doteq  \max\{1,\|E^\top\|/\rho\}$, the ADMM averaging matrix $M_\mathrm{avg} \doteq I - E^\top (E E^\top)^{-1} E$, 
$$
D_1 \doteq \begin{bmatrix} M_\mathrm{avg} \\ (E E^\top)^{-1} E \rho \end{bmatrix} \begin{bmatrix} I & I \end{bmatrix},
$$
and $d_1 \doteq \|D_1\|$.
Furthermore, given a constant $d_2$, define $c_2 \doteq d_1 + d_1 d_2 + d_2$.

\begin{lem}[ADMM convergence]\label{lem:admmConv}
	Suppose Assumption~\ref{ass:c3} holds and let $p^\star$ denote a KKT point which satisfies Assumption~\ref{ass:kkt}.
	Form QP~\eqref{eq:QP} at a primal-dual point $p^k = (z^k,\nu^k,\mu^k,\lambda^k)$ using the exact Hessian $H^k = \nabla_{zz}^2 L(z^k,\nu^k,\mu^k)$.
	Initialize Algorithm~\ref{alg:admm} with $z_i^k$ and $E_i^\top \lambda^k$ for all $i \in \mathcal S$. 
	Then, there exist an ADMM contraction factor $0 < a_w < 1$ and constants $d_2,\varepsilon > 0$ such that the following holds for all $0 < a < 1$.
	If $p^k \in \mathcal{B}(p^\star,\varepsilon)$ and if 
	\begin{equation}\label{eq:lmax}
	l_\mathrm{max} \geq 1 + \max \left\{0, \left\lceil \log_{a_w}\left(\frac{a}{c_1 c_2}\right) \right\rceil \right\},
	\end{equation} 
	then the iterates $p^{l_\text{max}} = (z^{l_\text{max}},\nu^{l_\text{max}},\mu^{l_\text{max}},\lambda^{l_\text{max}})$ returned by Algorithm~\ref{alg:admm} satisfy 
	\begin{equation*}
	\|p^{l_\text{max}} - p^{k,\star}_{\mathrm{QP}} \| \leq a \| p^k - p^{k,\star}_{\mathrm{QP}} \|.
	\end{equation*} 
	Furthermore, the active set stays constant, i.e., for all $i \in \mathcal{S}$
	\begin{align*}
	[h_i^k + \nabla h_i^{k\top}(y_i^l - z_i^k)]_{\mathcal{A}_i} &= 0 \quad \forall l \in \mathbb{N},\\
	[\mu_i^l]_{\mathcal{I}_i} &= 0 \quad \forall l \in \mathbb{N}.
	\end{align*}
\end{lem}
\hfill $\square$

\begin{proof}
	The proof proceeds in five steps.
	First, (a), we show that $p^{k,\star}_{\mathrm{QP}}$ is regular, if $p^k \approx p^\star$.
	Then, (b), ADMM converges q-linearly in the vector $w \doteq (z,\gamma/\rho)$, because QP~\eqref{eq:QP} is strictly convex over the constraint set.
	Then, (c), we apply the BST to the subsystem QPs and obtain that the active set is constant.
	Then, (d), we bound the error of the ADMM averaging step.
	Finally, (e), we derive the iteration bound~\eqref{eq:lmax}.
	
	(a) Let $\varepsilon \leq \varepsilon_1$.
	By the regularity of the KKT point $p^\star$, (Assumption~\ref{ass:kkt}), the solution to QP~\eqref{eq:QP} formed at $p^\star$ is also regular, i.e., if $p^k = p^\star$, then the QP solution $p^{k,\star}_{\mathrm{QP}}$ satisfies strict complementarity, LICQ, and the stronger SOSC
	\begin{equation}\label{eq:sSOSC}
	y^\top H^k y > 0 \text{ for all } y \neq 0 \text{ with } \nabla g(z^k)^\top y = 0.
	\end{equation}
	We can regard QP~\eqref{eq:QP} formed at $p^k \approx p^\star$ as a perturbed version of the QP formed at $p^\star$.
	Because of Assumption~\ref{ass:c3}, we can apply the BST~\cite[Thm. 3.2.2]{Fiacco1983} to QP~\eqref{eq:QP} and thus there exists $\varepsilon > 0$ such that the solution $p^{k,\star}_{\mathrm{QP}}$ to the perturbed QP~\eqref{eq:QP} is regular for all $p^k \in \mathcal{B}(p^\star,\varepsilon)$.
	The stronger SOSC condition~\eqref{eq:sSOSC} also holds if $p^k \in \mathcal{B}(p^\star,\varepsilon)$, which follows by adjusting the proof of the BST to the fact that Assumption~\ref{ass:kkt} (ii) is slightly stronger than SOSC~\cite[Lem. 3.2.1]{Fiacco1983}.

	(b) 
	From (a) we have that LICQ holds at $p^{k,\star}_{\mathrm{QP}}$ for all $p^k \in \mathcal{B}(p^\star,\varepsilon)$.
	Thus, the Jacobian $\nabla g_i^{k\top} \in \mathbb{R}^{n_{g_i} \times n_i}$ has full row rank for all $i \in \mathcal{S}$.
	Recall that for any $A \in \mathbb{R}^{m \times n}$, by the fundamental theorem of linear algebra, the null space of $A$ and the range space of $A^\top$ together form $\mathbb{R}^n$~\cite[p. 603]{Nocedal2006}.
	For all $i \in \mathcal{S}$, we can thus decompose any point $y_i \in \mathbb{R}^{n_i}$ via the null space method~\cite[Ch. 16]{Nocedal2006} as
	$$
	y_i = Z_i^kv_i + \nabla g_i^k w_i
	$$ with vectors $v_i \in \mathbb{R}^{n_i - n_{g_i}}$ and $w_i \in \mathbb{R}^{n_{g_i}}$, and a null space basis $Z_i^k \in \mathbb{R}^{n_i \times (n_i - n_{g_i})}$ of $\nabla g_i^{k\top}$.
	That is, $\nabla g_i^{k\top} Z_i^k = 0$.
	If $y_i \in \mathbb{Z}_i^k$, then $w_i$ is uniquely determined by the equality constraints~\eqref{eq:QPeq}: $w_i = - \left( \nabla g_i^{k\top}  \nabla g_i^k \right)^{-1} g_i^k$, which follows from the full row rank of $\nabla g_i^{k\top}$.
	Moreover, the stronger SOSC condition~\eqref{eq:sSOSC} implies that the reduced Hessian $\bar{H}_i^k \doteq Z_i^{k\top} H_i^k Z_i^k$ is positive definite for all $i \in \mathcal{S}$.
	Because $w_i$ is fixed for $y_i \in \mathbb{Z}_i^k$ and because $\bar{H}_i^k$ is positive definite, the objective $f_i^{\mathrm{QP},k}$ is strictly convex over the set $\mathbb{Z}_i^k$ for all $i \in \mathcal{S}$.
	Thus, $p^{k,\star}_{\mathrm{QP}}$ is the only KKT point of QP~\eqref{eq:QP}.
	Furthermore, the sets $\mathbb{Z}^k$ and $\mathbb{E}$ are closed and convex polyhedra, and they are feasible because $\varepsilon \leq \varepsilon_1$.
	Therefore, and because $f_i^{\mathrm{QP},k}$ are strictly convex over $\mathbb{Z}_i^k$ for all $i \in \mathcal{S}$, ADMM converges q-linearly in $w$ and, for all $p^k \in \mathcal{B}(p^\star,\varepsilon)$,
	\begin{equation}\label{eq:ADMMqlin}
	\| w^{l+1} - w^{k,\star}_{\mathrm{QP}} \| \leq a_w  \| w^{l} - w^{k,\star}_{\mathrm{QP}} \|
	\end{equation} for some $a_w < 1$ and for all $l \in \mathbb{N}_0$~\cite[Thm. 14]{Yang2016}.
	
	(c)
	The first ADMM step~\eqref{eq:admm1} solves a parametric QP, whose objective affinely depends on $(\rho{z}^l,\gamma^l) = \rho w^l$.
	The error $w^l - w_\mathrm{QP}^{k,\star}$ thus perturbs the QP in \eqref{eq:admm1}.	
	If~\eqref{eq:admm1} is parameterized with $w_\mathrm{QP}^{k,\star}$, then the step \eqref{eq:admm1} returns $y_\mathrm{QP}^{k,\star}$, $\nu_\mathrm{QP}^{k,\star}$, and $\mu_\mathrm{QP}^{k,\star}$.
	Thus, by the same arguments as in (a), the QP in step \eqref{eq:admm1} formed at $w_\mathrm{QP}^{k,\star}$ satisfies strict complementarity, LICQ, and the stronger SOSC condition~\eqref{eq:sSOSC} for all $p^k \in \mathcal{B}(p^\star,\varepsilon)$.
	We can therefore apply the BST to the perturbed subsystem QP in step \eqref{eq:admm1}.
	Thus, there exists $\varepsilon_3 > 0$ such that the map from $w^l - w^{k,\star}_{\mathrm{QP}}$ to $(y^{l+1},\nu^{l+1},\mu^{l+1})$ is continuously differentiable and the active set is constant for all $w^l \in \mathcal{B}(w^{k,\star}_{\mathrm{QP}},\varepsilon_3)$.
	That is, if $w^l \in \mathcal{B}(w^{k,\star}_{\mathrm{QP}},\varepsilon_3)$, then
	\begin{align*}
	[h_i^k + \nabla h_i^{k\top}(y_i^{l+1} - z_i^k)]_{\mathcal{A}_i} &= 0 \quad \forall i \in \mathcal{S} \\
	[\mu_i^{l+1}]_{\mathcal{I}_i} &= 0  \quad \forall i \in \mathcal{S}.
	\end{align*}
	Since local continuous differentiability implies local Lipschitz continuity, there exists a constant $d_2 < \infty$ such that, for all $w^l \in \mathcal{B}(w^{k,\star}_{\mathrm{QP}},\varepsilon_3)$ and for all $p^k \in \mathcal{B}(p^\star,\varepsilon)$,
	\begin{equation}\label{eq:contStep1}
	\left\| \begin{bmatrix} y^{l+1} \\ \nu^{l+1} \\ \mu^{l+1} \end{bmatrix}  - \begin{bmatrix} y^{k,\star}_{\mathrm{QP}} \\ \nu^{k,\star}_{\mathrm{QP}} \\ \mu^{k,\star}_{\mathrm{QP}} \end{bmatrix} \right\| \leq d_2 \| w^l - w^{k,\star}_{\mathrm{QP}} \|.
	\end{equation}
	Because of the q-linear ADMM convergence~\eqref{eq:ADMMqlin}, there exists $\varepsilon > 0$ such that $w^l \in \mathcal{B}(w^{k,\star}_{\mathrm{QP}},\varepsilon_3)$ for all $l \in \mathbb{N}_0$, if $p^k \in \mathcal{B}(p^\star,\varepsilon)$.

	(d)
	Define $\Delta y^{l} \doteq y^{l} - {y}^{k,\star}_{\mathrm{QP}}$, $\Delta z^{l} \doteq z^{l} - {z}^{k,\star}_{\mathrm{QP}}$, $\Delta \nu^{l} \doteq \nu^{l} - {\nu}^{k,\star}_{\mathrm{QP}}$, $\Delta \mu^{l} \doteq \mu^{l} - {\mu}^{k,\star}_{\mathrm{QP}}$, $\Delta \lambda^{l} \doteq \lambda^{l} - {\lambda}^{k,\star}_{\mathrm{QP}}$, and $\Delta \gamma^{l} \doteq \gamma^{l} - {\gamma}^{k,\star}_{\mathrm{QP}}$.
	The KKT conditions of the coordination QP in~\eqref{eq:admm2} yield
	\begin{align*}
	z^{l+1} &= M_{\mathrm{avg}} \left( y^{l+1} + \gamma^l/\rho\right) + E^\top \left( E E^\top \right)^{-1} c,\\
	\lambda^{l+1} &= \left( E E^\top  \right)^{-1} E \rho (y^{l+1} + \gamma^l / \rho) - \left( E E^\top \right)^{-1} \rho c
	\end{align*} for all $l \in \mathbb{N}_0$.
	Likewise,
	\begin{align*}
	z_{\mathrm{QP}}^{k,\star} &= M_{\mathrm{avg}} \left( y_{\mathrm{QP}}^{k,\star} + \gamma_{\mathrm{QP}}^{k,\star}/\rho\right) + E^\top \left( E E^\top \right)^{-1} c,\\
	\lambda_{\mathrm{QP}}^{k,\star} &= \left( E E^\top  \right)^{-1} E \rho (y_{\mathrm{QP}}^{k,\star} + \gamma_{\mathrm{QP}}^{k,\star} / \rho) - \left( E E^\top \right)^{-1} \rho c
	\end{align*}
	and hence
	\begin{equation}\label{eq:admmAvgError}
	\begin{bmatrix} \Delta z^{l+1} \\ \Delta \lambda^{l+1} \end{bmatrix} = \underbrace{\begin{bmatrix} M_\mathrm{avg} \\ (E E^\top)^{-1} E \rho \end{bmatrix} \begin{bmatrix} I & I \end{bmatrix}}_{D_1} \begin{bmatrix} \Delta y^{l+1} \\ \Delta \gamma^l / \rho \end{bmatrix}.
	\end{equation}
	Since $d_1 = \|D_1\|$, we obtain, for all $l \in \mathbb{N}_0$,
	\begin{equation}\label{eq:contStep2}
	\left\| \begin{bmatrix} \Delta z^{l+1} \\ \Delta \lambda^{l+1} \end{bmatrix} \right\| \leq  d_1 \left\| \begin{bmatrix} \Delta y^{l+1} \\ \Delta \gamma^l / \rho \end{bmatrix} \right\|.
	\end{equation}
	
	(e) 
	Combining the continuity properties~\eqref{eq:contStep1} and~\eqref{eq:contStep2} yields, for all $p^k \in \mathcal{B}(p^\star,\varepsilon)$ and for all $l \in \mathbb{N}_0$,
	\begin{align}
	\left\| p^{l+1} - p^{k,\star}_{\mathrm{QP}} \right\| &=\nonumber  \left\| \begin{bmatrix} \Delta{z}^{l+1} \\ \Delta \nu^{l+1} \\ \Delta \mu^{l+1} \\ \Delta \lambda^{l+1} \end{bmatrix} \right\| \\
	&\nonumber \leq \left\| \begin{bmatrix} \Delta \nu^{l+1} \\ \Delta \mu^{l+1} \end{bmatrix} \right \| + \left\| \begin{bmatrix} \Delta z^{l+1} \\ \Delta \lambda^{l+1} \end{bmatrix} \right\| \\
	&\nonumber \leq d_2 \| w^l - w^{k,\star}_{\mathrm{QP}} \| + d_1 \left\| \begin{bmatrix} \Delta y^{l+1} \\ \Delta \gamma^{l} / \rho \end{bmatrix} \right\| \\
	&\nonumber \leq d_2 \| w^l - w^{k,\star}_{\mathrm{QP}} \| + d_1  \| \Delta y^{l+1}\| + d_1 \| \Delta \gamma^{l} / \rho \|\\
	&\label{eq:Lipschitz_p} \leq \underbrace{\left( d_1 + d_1d_2 + d_2  \right)}_{= c_2} \| w^l - w^{k,\star}_{\mathrm{QP}} \|.
	\end{align}
	
	The KKT system of QP~\eqref{eq:QPadmm} implies $\gamma^{k,\star}_{\mathrm{QP}} = E^\top \lambda^{k,\star}_{\mathrm{QP}}$ and the KKT system of the coupling QP in Step~\eqref{eq:admm2} yields $\rho (z^{l+1}-y^{l+1}) - \gamma^l + E^\top \lambda^{l+1} = 0$.
	Inserting the last equation into the ADMM dual update~\eqref{eq:admm3} gives $\gamma^{l+1} = E^\top \lambda^{l+1}$ for all $l \in \mathbb{N}_0$.
	Combining this with the definitions of $w$ and $p$~and with the initialization $\gamma_i^0 = E_i^\top \lambda^0$ for $k = 0$ yields, for all $l \in \mathbb{N}_0$,
	\begin{equation*}
	w^l - w^{k,\star}_{\mathrm{QP}} = \begin{bmatrix} I & 0 & 0 & 0 \\ 0 & 0 & 0 & E^\top/\rho \end{bmatrix} \left( p^l - p^{k,\star}_{\mathrm{QP}} \right)
	\end{equation*} 	
	and hence
	\begin{equation}\label{eq:Lipschitz_w}
	\left\| w^l - w^{k,\star}_{\mathrm{QP}} \right\| \leq \underbrace{\left\| \begin{bmatrix} I & 0 & 0 & 0 \\ 0 & 0 & 0 & E^\top/\rho \end{bmatrix} \right\|}_{c_1} \left\| p^l - p^{k,\star}_{\mathrm{QP}} \right\|.
	\end{equation}
	Combining the q-linear convergence~\eqref{eq:ADMMqlin} with the Lipschitz bounds~\eqref{eq:Lipschitz_p}--\eqref{eq:Lipschitz_w} yields, for all $p^k \in \mathcal{B}(p^\star,\varepsilon)$,
	\begin{align*}
	\| p^{k+1} - p^{k,\star}_{\mathrm{QP}} \| & = \| p^{l_\text{max}} - p^{k,\star}_{\mathrm{QP}} \| \\
	& \leq c_2 \|  w^{l_\text{max}-1} - w^{k,\star}_{\mathrm{QP}} \|\\
	& \leq c_2 (a_w)^{l_\text{max}-1} \| w^k - w^{k,\star}_{\mathrm{QP}} \|\\
	& \leq \underbrace{c_1 c_2 (a_w)^{l_\text{max}-1}}_{\stackrel{!}{=}a} \| p^k - p^{k,\star}_{\mathrm{QP}} \|.
	\end{align*}
	Rearranging 
	$ a = c_1 c_2 (a_w)^{l_\mathrm{max}-1}$ yields
	$$
	l_\mathrm{max} = 1 + \log_{a_w}\left( \frac{a}{c_1c_2} \right).
	$$
	The algorithm is only well-defined if $l_\mathrm{max} \geq 1$ and thus we obtain the inner iteration bound~\eqref{eq:lmax}.
\end{proof}

The bound~\eqref{eq:lmax} depends on the ADMM contraction factor $a_w$ in~\eqref{eq:ADMMqlin}, the constant $d_2$ in the subsystem QP error~\eqref{eq:contStep1}, and $a$.
The constants $a_w$, $c_1$, and $c_2$ depend on the penalty parameter $\rho$ and can thus be tuned for a given problem.
The design parameter $a \in (0,1)$ controls the accuracy to which ADMM solves QP~\eqref{eq:QP}: choosing $a$ small increases $l_\mathrm{max}$ in~\eqref{eq:lmax}, but also allows for a larger sampling interval $\bar{\delta}$ in the stability guarantee.

The q-linear convergence~\eqref{eq:ADMMqlin} of ADMM is global, i.e., it holds for any $w^l \in \mathbb{R}^{2n}$.
To the best of our knowledge, there does not yet exist a quantification of the contraction factor $a_w$ for the global ADMM convergence that would be applicable here.
However, we can quantify $a_w$ and $d_2$ for $w^l \approx w_\mathrm{QP}^{k,\star}$ inside the area of constant active set as follows.

Denote the centralized active and inactive sets as $\mathcal{A} \doteq \{ j \in \{ 1,\dots,n_h \}\; | \; [h(z^\star)]_j = 0\}$ and $\mathcal{I} \doteq \{j \in \{1,\dots,n_h \} \; | \; [h(z^\star)]_j <0 \}$, resepctively.
Furthermore, define the shorthands $h_\mathcal{A}^k \doteq [h(z^k)]_\mathcal{A}$, $\nabla h^{k\top}_\mathcal{A} \doteq [\nabla h(z^k)^{\top}]_{\mathcal{A}}$, $\mu_\mathcal{A} \doteq [\mu]_\mathcal{A}$, and $\mu_\mathcal{I} \doteq [\mu]_\mathcal{I}$.
The centralized KKT matrix of the subsystem QPs in Step~\eqref{eq:admm1} is regular for all $p^k \in \mathcal{B}(p^\star,\varepsilon)$ with $\varepsilon$ from Lemma~\ref{lem:admmConv} and reads
\begin{equation*}
K^k \doteq \begin{bmatrix}
	H^k + \rho I & \nabla g^k & \nabla h_{\mathcal{A}}^k \\
	\nabla g^{k\top} & 0 & 0\\
	\nabla h_\mathcal{A}^{k\top} & 0 & 0 
	\end{bmatrix}.
\end{equation*} 
For all $p^k \in \mathcal{B}(p^\star,\varepsilon)$, we define the matrix
\begin{equation}\label{eq:D2}
D^k \doteq \rho \cdot (K^k)^{-1} \begin{bmatrix} I & -I \\ 0 & \phantom{-}0 \\ 0 & \phantom{-}0 \end{bmatrix}.
\end{equation}

\begin{lem}[Lipschitz constant of the subsystem QP]\label{lem:d2}
	Suppose Assumption~\ref{ass:c3} holds and let $p^\star$ denote a KKT point which satisfies Assumption~\ref{ass:kkt}.
	Form QP~\eqref{eq:QP} at a primal-dual point $p^k$ using the exact Hessian $H^k = \nabla_{zz}^2 L(z^k,\nu^k,\mu^k)$.
	Initialize Algorithm~\ref{alg:admm} with $z_i^k$ and $E_i^\top \lambda^k$ for all $i \in \mathcal{S}$.
	Let $p^k \in \mathcal{B}(p^\star,\varepsilon)$ with $\varepsilon$ from Lemma~\ref{lem:admmConv}.
	Then, the Lipschitz constant $d_2$ in~\eqref{eq:contStep1} is given by
	\begin{equation}\label{eq:d2}
	d_2 = \max_{p^k \in \mathcal{B}(p^\star,\varepsilon)} \left\|D^k\right\|. 
	\end{equation}
\end{lem}~\hfill $\square$

\begin{proof}
As discussed in the proof of Lemma~\ref{lem:admmConv}, the active sets of NLP~\eqref{eq:sepForm}, QP~\eqref{eq:QP}, and the centralized subsystem QP in Step~\eqref{eq:admm1} are equivalent for $p^k \in \mathcal{B}(p^\star,\varepsilon)$.
Moreover, strict complementarity, LICQ, and the stronger SOSC version~\eqref{eq:sSOSC} hold at the subsystem QP solution $(y^{l+1},\mu^{l+1})$ for all ${l \in \mathbb{N}_0}$.
In centralized form, the KKT system of the subsystem QPs in Step~\eqref{eq:admm1} reads
\begin{align*}
\underbrace{\begin{bmatrix}
	H^k + \rho I & \nabla g^k & \nabla h_{\mathcal{A}}^k \\
	\nabla g^{k\top} & 0 & 0\\
	\nabla h^{k\top}_{\mathcal{A}} & 0 & 0 
	\end{bmatrix}}_{= K^k} \begin{bmatrix} y^{l+1} \\ \nu^{l+1} \\ \mu_{\mathcal{A}}^{l+1} \end{bmatrix} = \begin{bmatrix} -\nabla f^k - \gamma^l + \rho z^l \\ -g^k + \nabla g^{k\top} z^k \\ -h_\mathcal{A}^k + \nabla h_\mathcal{A}^{k\top} z^k \end{bmatrix}.
\end{align*}
Furthermore, $\mu_{\mathcal{I}}^{l+1} = 0$.
The KKT matrix $K^k$ is invertible for all $p^k \in \mathcal{B}(p^\star,\varepsilon)$, because of LICQ and the stronger SOSC condition~\eqref{eq:sSOSC}, cf.~\cite[Lem. 16.1]{Nocedal2006}.
Rearranging thus yields
\begin{align*}
\begin{bmatrix} y^{l+1} \\ \nu^{l+1} \\ \mu_{\mathcal{A}}^{l+1} \end{bmatrix} = \left(K^k\right)^{-1} \begin{bmatrix} -\nabla f^k -\gamma^l + \rho z^l \\ -g^k + \nabla g^{k\top} z^k \\ -h_\mathcal{A}^k + \nabla h_\mathcal{A}^{k\top} z^k \end{bmatrix}.
\end{align*} Likewise,
\begin{align*}
\begin{bmatrix} y_{\mathrm{QP}}^{k,\star} \\ \nu_{\mathrm{QP}}^{k,\star} \\ \mu_{\mathcal{A},\mathrm{QP}}^{k,\star} \end{bmatrix} = \left(K^k\right)^{-1} \begin{bmatrix} -\nabla f^k -\gamma_{\mathrm{QP}}^{k,\star} + \rho z_{\mathrm{QP}}^{k,\star} \\ -g^k + \nabla g^{k\top} z^k \\ -h^k_\mathcal{A} + \nabla h_\mathcal{A}^{k\top} z^k \end{bmatrix}
\end{align*} and hence
\begin{align}\label{eq:subsyserror}
\begin{bmatrix} \Delta y^{l+1} \\ \Delta \nu^{l+1} \\ \Delta \mu_{\mathcal{A}}^{l+1} \end{bmatrix} = \underbrace{ \rho \cdot \left(K^k\right)^{{-1}} \begin{bmatrix} I & -I \\ 0 & \phantom{-}0 \\ 0 & \phantom{-}0 \end{bmatrix} }_{= D^k} \Delta w^l.
\end{align}
Since $\mu_\mathcal{I}^{l+1} = \mu_{\mathcal{I},\mathrm{QP}}^{k,\star} = 0$, we obtain~\eqref{eq:contStep1} with $d_2$ given in~\eqref{eq:d2}.
\end{proof}

Lemma~\ref{lem:d2} shows how to compute the constant $d_2$ which is required to evaluate $c_2$ in the sufficient iteration bound~\eqref{eq:lmax}.
We next compute the ADMM contraction factor $a_w$ by showing that, inside $\mathcal{B}(p^\star,\varepsilon)$, ADMM behaves like a stable Linear Time-Invariant~(LTI) system.
To this end, we first partition the top block rows of $D^k$ defined in~\eqref{eq:D2} as
\begin{equation*}\label{eq:defT}
\begin{bmatrix} I & 0 & 0 \end{bmatrix} D^k \doteq \begin{bmatrix} T^k & -T^k \end{bmatrix}
\end{equation*} to obtain the block matrix $T^k \in \mathbb{R}^{n \times n}$.
Then, we define the ADMM system matrix
\begin{equation*}
A^k \doteq \begin{bmatrix} M_\mathrm{avg}T^k & M_\mathrm{avg} (I - T^k) \\
(I-M_\mathrm{avg}) T^k & (I-M_\mathrm{avg}) (I - T^k)  \end{bmatrix}.
\end{equation*}

\begin{lem}[ADMM contraction factor]\label{lem:a1}
Suppose Assumption~\ref{ass:c3} holds and let $p^\star$ denote a KKT point which satisfies Assumption~\ref{ass:kkt}.
Form QP~\eqref{eq:QP} at a primal-dual point $p^k$ using the exact Hessian $H^k = \nabla_{zz}^2 L(z^k,\nu^k,\mu^k)$.
Initialize Algorithm~\ref{alg:admm} with $z_i^k$ and $E_i^\top \lambda^k$ for all $i \in \mathcal{S}$.
Let $p^k \in \mathcal{B}(p^\star,\varepsilon)$ with $\varepsilon$ from Lemma~\ref{lem:admmConv}.
Then, the ADMM iterations read
\begin{equation}\label{eq:admmlti}
w^{l+1} - w^{k,\star}_{\mathrm{QP}} = A^k \left(w^l - w^{k,\star}_{\mathrm{QP}}\right).
\end{equation} Furthermore, the ADMM contraction factor $a_w < 1$ in~\eqref{eq:ADMMqlin} is given by
\begin{equation}\label{eq:a1}
a_w = \max_{p^k \in \mathcal{B}(p^\star,\varepsilon)} \left\|A^k\right\|.
\end{equation}
~\hfill $\square$
\end{lem}
\begin{proof}
Lemma~\ref{lem:admmConv} guarantees that the active set of the centralized subsystem QP stays constant for all $l \in \mathbb{N}_0$, if $p^k \in \mathcal{B}(p^\star,\varepsilon)$.
Thus, the subsystem QP error is given by~\eqref{eq:subsyserror}.
Rewriting the top block rows of~\eqref{eq:subsyserror} yields
\begin{equation}\label{eq:subsyserror2}
\Delta y^{l+1} = \begin{bmatrix} T^k & -T^k\end{bmatrix} \Delta w^l.
\end{equation}
Recall the error in the ADMM averaging step from~\eqref{eq:admmAvgError},
\begin{equation}\label{eq:admmAvgError2}
\Delta z^{l+1} = M_\mathrm{avg} \left( \Delta y^{l+1} + \Delta \gamma^l / \rho \right).
\end{equation} Inserting~\eqref{eq:subsyserror2} into~\eqref{eq:admmAvgError2} yields
\begin{equation}\label{eq:admmAvgError3}
\Delta z^{l+1} = \begin{bmatrix} M_\mathrm{avg} T^k & M_\mathrm{avg} (I-T^k) \end{bmatrix} \Delta w^l.
\end{equation} By rearranging the dual update~\eqref{eq:admm3}, we obtain
\begin{equation*}
\Delta \gamma^{l+1} / \rho = \Delta \gamma^l / \rho + \left( \Delta y^{l+1} - \Delta z^{l+1} \right)
\end{equation*} and hence
\begin{equation}\label{eq:gammaerror}
\frac{\Delta \gamma^{l+1}}{\rho} = \begin{bmatrix} (I - M_\mathrm{avg})T^k & (I - M_\mathrm{avg}) (I - T^k) \end{bmatrix} \Delta w^l.
\end{equation} By combining~\eqref{eq:admmAvgError3} and~\eqref{eq:gammaerror}, we can write ADMM as the LTI system
\begin{equation*}
\Delta w^{l+1} = \underbrace{\begin{bmatrix} M_\mathrm{avg}T^k & M_\mathrm{avg} (I - T^k) \\
(I-M_\mathrm{avg}) T^k & (I-M_\mathrm{avg}) (I - T^k)  \end{bmatrix}}_{A^k} \Delta w^l
\end{equation*} to obtain~\eqref{eq:admmlti}.
Moreover, 
\begin{equation*}
\left\| \Delta w^{l+1} \right\| \leq \left\| A^k \right\| \left\| \Delta w^l \right\|.
\end{equation*} The q-linear ADMM convergence~\eqref{eq:ADMMqlin} yields
\begin{equation*}
\left\| A^k \right\| = \max_{\Delta w^l \neq 0} \frac{\left\| A^k \Delta w^{l} \right\|}{\left\| \Delta w^l \right\| } = \max_{\Delta w^l \neq 0} \frac{\left\| \Delta w^{l+1} \right\|}{\left\| \Delta w^l \right\| } < 1
\end{equation*} for all $p^k \in \mathcal{B}(p^\star,\varepsilon)$.
Finally, we obtain the worst-case ADMM contraction factor $a_w$ inside $\mathcal{B}(p^\star,\varepsilon)$ via~\eqref{eq:a1}.
\end{proof}

\begin{rem}[Multiple ADMM iterations per QP]
	Convergence of the inexact SQP scheme is guaranteed if the approximate QP solutions satisfy~\eqref{eq:A1}. 
	This condition requires a sufficiently accurate guess for the primal-dual variables~$p$.
	While ADMM is q-convergent in $w$, it only is r-convergent in $p$.
	Hence in applications multiple ADMM iterations per SQP step are required to guarantee~\eqref{eq:A1}. \hfill $\square$  
\end{rem}

\subsection{Optimizer Convergence with Limited Inner Iterations}

We now combine the outer convergence from Lemma~\ref{lem:truncSQP} with the inner convergence from Lemma~\ref{lem:admmConv} to prove the q-linear convergence of Algorithm~\ref{alg:d-SQP} for fixed $l_\text{max}$.

\begin{thm}[dSQP convergence]\label{thm:dsqp}
	Suppose Assumption~\ref{ass:c3} holds and let $p^\star$ denote a KKT point which satisfies Assumption~\ref{ass:kkt}.
	Form QP~\eqref{eq:QP} with the exact Hessian $H^k = \nabla_{zz}^2 L(z^k,\nu^k,\mu^k)$.
	Let $0 < a < \bar{a}_p < 1$.
	Then there exists a constant $\varepsilon > 0$ such that the following holds with $d_2$ and $a_w$ computed via~\eqref{eq:d2} and~\eqref{eq:a1}.
	 
	If $p^0 \in \mathcal{B}(p^\star,\varepsilon)$ and if the number $l_\mathrm{max}$ of ADMM iterations per SQP iterations satisfies~\eqref{eq:lmax}, then the sequence $\{p^k\}$ generated by Algorithm~\ref{alg:d-SQP} converges to $p^\star$ and, for all $k \in \mathbb{N}_0$, satisfies
	$\| p^{k+1} - p^\star \| \leq \bar{a}_p \| p^k - p^\star \|.$ 
	\hfill $\square$
	\end{thm}
\begin{proof}
	By the outer convergence result (Lemma~\ref{lem:truncSQP}), there exists $\varepsilon_2 > 0$ such that the following holds. 
	If the ADMM solutions to QP~\eqref{eq:QP} satisfy the sufficient accuracy requirement~\eqref{eq:A1} for the inexact SQP scheme and if $p^0 \in \mathcal{B}(p^\star,\varepsilon_2)$, then the sequence $\{p^k\}$ generated by Algorithm~\ref{alg:d-SQP} converges q-linearly to $p^\star$.
	Let $\varepsilon \leq \varepsilon_2$.
	From the inner convergence (Lemma~\ref{lem:admmConv}), there exists a radius $\varepsilon > 0$ such that the ADMM solutions satisfy~\eqref{eq:A1} for all $p^k \in \mathcal{B}(p^\star,\varepsilon)$.
	Thus, we obtain q-linear convergence in the outer iterations.
\end{proof}

\begin{rem}[Relation to~\cite{Stomberg2022a}]
	An earlier version of dSQP is presented in~\cite{Stomberg2022a}.
	There, an inexact Newton-type stopping criterion terminates ADMM dynamically in each SQP step and control applications were not considered.
	Theorem~\ref{thm:dsqp} guarantees convergence when the number of ADMM iterations per SQP iteration is fixed to $l_\text{max}$.
	This avoids the online communication of convergence flags and allows us to use dSQP in the proposed decentralized RTI scheme.
	Consequently, the convergence proof derived in this article differs from~\cite{Stomberg2022a} and, inspired by a centralized inexact SQP scheme from~\cite{Zanelli2021b}, is centered around inequality~\eqref{eq:A1} instead of inexact Newton methods. \hfill $\square$
\end{rem}

\section{Decentralized Real-Time Iterations}\label{sec:drti}

We now combine the q-linear convergence of dSQP and the RTI stability result from Lemma~\ref{lem:rti}.

Consider the setpoint stabilization for the network $\mathcal S$ of dynamical systems with dynamics~\eqref{docp:sys}.
In each control sampling interval $\delta$, the state $x(t)$ is assumed to be measured and a constant input signal is applied to the system.
We define the distributed NMPC control law $\kappa_\mathrm{d} : \mathbb{X}_0 \times \mathbb{R}^{n_p} \rightarrow \mathbb{R}^{n_u}$ as the map from the centralized state $x(t)$ and the dSQP initialization~$p^0(t)$ to the first part $\bar{u}^{k_\text{max}}(0)=M_{u,p}p^{k_\mathrm{max}}(t)$ of the centralized input trajectory which is returned by Algorithm~\ref{alg:d-SQP} with settings $k_\text{max} \in \mathbb{N}$ and $H_i = \nabla_{z_iz_i}^2 L_i$.
That is, in each NMPC step we apply dSQP with one or more outer iterations and with $l_\text{max}$ inner iterations per outer iteration to OCP~\eqref{docp}, or respectively its NLP reformulation~\eqref{eq:sepForm}.
Furthermore, we initialize dSQP with the OCP solution obtained in the previous NMPC step, i.e. $p^0(t) = p^{k_\mathrm{max}}(t-1)$.
Thus, the centralized system~\eqref{ocp:sys} and dSQP form the closed-loop system optimizer dynamics~\eqref{eq:sysoptdyn}, where $\Phi$ maps the current state $x(t)$ and dSQP output $p^{k_\mathrm{max}}(t)$ to the dSQP solution at the next time step.
We are now ready to state our main result.

\begin{thm}[Decentralized RTI stability]\label{thm:drti}
	Suppose that Assumptions~\ref{ass:Lyapunov} and \ref{ass:Lipschitz_sys} hold and that $\bar{p}(0) = 0$.
	Further assume, for all $x \in \mathbb{X}_{\bar V}$, that Assumption~\ref{ass:c3} holds and that $\bar{p}(x)$ satisfies Assumption~\ref{ass:kkt}.
	Consider a constant $a \in (0,1)$ for Theorem~\ref{thm:dsqp}. 
	For all $x \in \mathbb{X}_{\bar V}$, denote the dSQP convergence radius and contraction factor from Theorem~\ref{thm:dsqp} by $\varepsilon(x)$ and $\bar{a}_p(x)$, compute the constants $d_2(x)$ and $a_w(x)$ according to Lemmas~\ref{lem:d2} and~\ref{lem:a1}, and assume the number of inner iterations $l_\mathrm{max}$ satisfies~\eqref{eq:lmax}.
	Let $\hat{r}_p = \min_{x \in \mathbb{X}_{\bar V}} \varepsilon(x)$, $a_p = \max_{x\in \mathbb{X}_{\bar{V}}} \bar{a}_p(x)$, and compute $\bar{\delta}$ and $\tilde{r}_p$ via~\eqref{eq:RTIsuff}.
	Then, the following holds.

	If the control sampling interval satisfies $\delta \leq \bar{\delta}$,
	then the origin is a locally exponentially stable equilibrium with region of attraction	
	\begin{equation*}
	\Sigma = \left\{ (x,p^{k_\mathrm{max}}) \in \mathbb{R}^{n_x + n_p} \left| \, x \in \mathbb{X}_{\bar{V}}, \left\| p^{k_\mathrm{max}} - \bar{p}(x) \right\| \leq \tilde r_p \right. \right\}
	\end{equation*}	
	of the closed-loop system-optimizer dynamics~\eqref{eq:sysoptdyn}. \hfill $\square$
\end{thm}
\begin{proof}
	Consider the reformulation of OCP~\eqref{docp} as a partially separable NLP~\eqref{eq:sepForm}.
	According to the dSQP convergence Theorem~\ref{thm:dsqp}, Algorithm~\ref{alg:d-SQP} converges q-linearly to the primal-dual solution $\bar p(x(t))$ for any initialization $p^0(t) \in \mathcal{B}(\bar p(x(t)),\hat{r}_p)$ and for all $x \in \mathbb{X}_{\bar{V}}$.
	Hence, dSQP satisfies the q-linear convergence Assumption~\ref{ass:qlin}.
	The Lipschitz continuity of the map $\bar p$ follows from Assumption~\ref{ass:kkt} via the BST~\cite[Thm. 3.2.2]{Fiacco1983}, because the functions in NLP~\eqref{eq:sepForm} are twice continuously differentiable and because $\nabla g$ is continously differentiable with respect to $x(t)$.
	That is, Assumption~\ref{ass:Lipschitz_z} holds.
	The exponential stability of the origin for the closed-loop system-optimizer dynamics therefore follows from Lemma~\ref{lem:rti}.
\end{proof}

\begin{rem}[Region of attraction]\label{rem:roa}
	If the initial state $x(0)$ and the iterate~$p^{k_\mathrm{max}}(0)$ returned by dSQP in the first NMPC step lie inside the region of attraction, i.e. if ${(x(0),p^{k_\mathrm{max}}(0)) \in \Sigma}$, then the system-optimizer dynamics converge to the origin.
	The set $\mathbb{X}_{\bar{V}} = \{x \in \mathbb{R}^{n_x} \; | \; V(x) \leq \bar{V}\}$ is the level set over which Assumptions~\ref{ass:Lyapunov},~\ref{ass:Lipschitz_sys},~\ref{ass:c3}, and~\ref{ass:kkt} hold and $\tilde{r}_p$ in the definition of $\Sigma$ can be computed via~\eqref{eq:RTIsuff}. \hfill $\square$
\end{rem}

\begin{rem}[Continuity of $V$, $\sqrt{V}$, and $\bar{p}$]\label{rem:Lipschitz}
	Assumption~\ref{ass:Lyapunov} requires that $V$ is continuous for all $x \in \mathbb{X}_{\bar V}$ and that $\sqrt{V}$ is Lipschitz continuous at the origin.
	These conditions hold if $V$~is Lipschitz continuous over $\mathbb{X}_{\bar{V}}$ and twice continuously differentiable at the origin~\cite{Zanelli2021}. This is the case if Assumption~\ref{ass:kkt} holds for all $x \in \mathbb{X}_{\bar V}$~\cite[Thm. 3.4.1]{Fiacco1983}.
	The Lipschitz continuity of $\bar{p}$ in Assumption~\ref{ass:Lipschitz_z} follows from the differentiability Assumption~\ref{ass:c3} and the regularity Assumption~\ref{ass:kkt} via~\cite[Thm. 3.2.2]{Fiacco1983}. \hfill $\square$	
\end{rem}

\begin{rem}[Stability close to global optima]
	The careful reader will have noticed that the stability result of Theorem~\ref{thm:drti} requires the optimizer initialization to be close to a global minimum of the OCP.
	This assumption carries over from the centralized result~\cite[Thm. 25]{Zanelli2021}, which we invoke to prove stability.
	In practice, one will only obtain local minima for non-convex NLPs.
	However, we observe stability in simulations despite the convergence to local minima.	
	A key challenge when guaranteeing stability of MPC under local minima is that the OCP value function $V$ is unknown. Suboptimal MPC addresses this challenge for OCPs with terminal constraints by exploiting the feasibility of local minima~\cite[Sec. 2.7]{Rawlings2019}. However, RTI schemes typically favor rapid computation at the cost of asymptotic feasibility. Thus, and to the best of our knowledge, centralized and decentralized RTI stability guarantees which only require optimizer initializations close to local minima are yet unavailable.
	\hfill $\square$
\end{rem}

\begin{rem}[Relation to~\cite{Hours2016}]\label{rem:Hours}
	Real-time distributed NMPC is also addressed in~\cite{Hours2016} and we comment on similarities and differences to our approach.
	Both schemes share three important properties:
	First, they do not require a coordinator.
	Second, they apply an a-priori fixed number of optimizer iterations in each NMPC step to find a suboptimal control input.
	Third, the optimizer contraction for subsequent NMPC steps is proven via two key ingredients: (a) q-linear optimizer convergence for sufficiently many inner iterations, i.e., $l_\text{max}$ ADMM iterations per SQP step for dSQP and $M$ primal iterations for Algorithm~1 in~\cite{Hours2016}, where $M$ is in the notation of~\cite{Hours2016}, and (b), sufficient proximity between subsequent state measurements, i.e., by choosing the sampling interval sufficiently small or by direct assumption~\cite{Hours2016}.
	
	On the other hand, the schemes differ in the employed optimization algorithm, in the order in which computations are executed on the subsystems, and in the obtained convergence result.
	The computationally expensive Step~\ref{admm-step:1} in ADMM can be executed by all subsystems in parallel. In contrast, the decomposition scheme in~\cite{Hours2016} assigns subsystems into groups and the groups execute computation steps in sequence.
	And whereas Theorem~\ref{thm:drti} proves the local exponential stability of system and optimizer, \cite[Thm. 5]{Hours2016} bounds the optimizer suboptimality in closed-loop without addressing the system asymptotics.
	However, we invoke the more recent stability results from~\cite{Zanelli2021} to prove stability. 
	Due to the q-linear optimizer convergence of Algorithm~1 of~\cite{Hours2016}, we conjecture that Theorem~\ref{thm:drti} also holds under suitable assumptions if dSQP is replaced  with the approach from~\cite{Hours2016}. \hfill $\square$
\end{rem}

\begin{rem}[Application to linear-quadratic DMPC]
	If NLP~\eqref{eq:sepForm} is a convex QP, dSQP reduces to ADMM.
	In this case, the q-linear convergence of ADMM yields closed-loop stability for $l_\mathrm{max} \geq 1$ and sampling intervals below~$\bar{\delta}$.
	This follows from Lemma~\ref{lem:rti} by replacing the convergence in $p$ in Assumption~\ref{ass:qlin} with the ADMM convergence~\eqref{eq:ADMMqlin}, cf.~\cite{Zanelli2021b}. \hfill $\square$
\end{rem}

\subsection{Implementation Aspects}

We next comment on the communication requirements of the proposed RTI scheme and on the choice of a suitable Hessian for constructing QP~\eqref{eq:QP}.

With respect to communication, the decentralized RTI scheme requires each subsystem to exchange messages with all neighbors in order to execute Step~\ref{admm-step2} of Algorithm~\ref{alg:admm}.
All other steps in ADMM and dSQP can be carried out without communication.
As mentioned in Remark~\ref{rem:averaging}, this decentralized ADMM implementation is due to the fact that Step~\ref{admm-step2} is equivalent to an averaging procedure~\cite[Ch. 7]{Boyd2011}.
Essentially, the step requires each subsystem to exchange predicted state trajectories with neighbors, compute an averaged state trajectory, and then to exchange the averaged trajectories with neighbors.
Details about the implementation of the averaging procedure in a DMPC context are provided in~\cite{Bestler2019,Stomberg2023}.
In this context, we comment on the relation between the dual variables $\lambda$ and $\gamma$.
Observe that $\lambda$ is implicitly updated by ADMM in~\eqref{eq:admm2}, but only appears in the initialization Step~\ref{dsqp-stp:1} of Algorithm~\ref{alg:d-SQP}. 
Apart from this initialization, $\lambda$ is only needed in the theoretical convergence analysis and not in the implementation of dSQP.
This facilitates the averaging step for computing $z^{l+1}$ as outlined in Appendix~\ref{app:example}, where the explicit computation of $\lambda^{l+1}$ is omitted. 
ADMM ensures $\gamma_i^{l+1} = E_i^\top \lambda^{l+1}$ for all $i \in \mathcal{S}$ and for all $k,l \in \mathbb{N}_0$.
This can be used in closed-loop control to warm start $\gamma_i^0$ in Step~\ref{dsqp-stp:1} of dSQP with $\gamma_i^{k_\mathrm{max}}$ from the previous NMPC step.
Thus, an initialization $\lambda^0$ is only required in the first control step.
Possible initializations for the first control step that often work well are $\lambda^0 = 0$ or $\lambda^0 = \lambda^\star$, if available.

With respect to the Hessian, the stability analysis holds for the exact Hessian $H_i = \nabla_{z_iz_i}^2 L_i$, because this guarantees local q-linear convergence of dSQP such that the controller meets Assumption~\ref{ass:qlin}.
However, the exact Hessian may not always be a favorable choice, because $\nabla_{z_iz_i}^2 L_i$ can be indefinite outside the dSQP convergence region $\mathcal{B}(p^\star,\varepsilon)$ and because $\nabla_{z_iz_i}^2 L_i$ must be evaluated in each dSQP iteration.
Instead, the Constrained Gauss-Newton (CGN) method, first introduced as the generalized Gauss-Newton method~\cite{Bock1983}, provides an alternative which is often effective in NMPC~\cite{Messerer2021}.
Consider a specialized version of NLP~\eqref{eq:sepForm} with least squares objectives
\begin{equation}\label{eq:GNobj}
f_i(z_i) = \frac{1}{2} \| M_i z_i - m_i \|_2^2,
\end{equation} where the matrices $M_i \in \mathbb{R}^{n_i \times n_i}$ are positive definite and the vectors $m_i \in \mathbb{R}^{n_i}$ are constants for all $i \in \mathcal{S}$. 
Such objectives commonly occur in NMPC for setpoint stabilization, see Section~\ref{sec:numer}.
The CGN method builds QP~\eqref{eq:QP} with the Gauss-Newton (GN) Hessian approximation $H_i = B_i \doteq M_i^\top M_i$ for all $i \in \mathcal{S}$.
Crucially, the matrices $B_i$ are constant, positive definite, and can be evaluated offline.

However, stability guarantees for our proposed RTI scheme when using the GN Hessian are yet unavailable.
While the CGN method is locally guaranteed to converge q-linearly in some norm~\cite{Messerer2021}, the convergence is not necessarily q-linear in the Euclidean norm $\|p^k - p^\star\|$.
This impedes a straight-forward stability proof via Lemma~\ref{lem:rti} as Assumption~\ref{ass:qlin} may not hold.

\section{Numerical Results}\label{sec:numer}

We consider the setpoint stabilization of coupled inverted pendulums. Each pendulum is attached to a cart and the carts are coupled via springs as shown in Figure~\ref{fig:pendulum}.
Let $q_i$ be the cart position and $\varphi_i$ be the angular deviation from the upright position for pendulum $i$.
The state of pendulum $i$ is $x_i = (q_i, \dot q_i, \varphi_i, \dot \varphi_i) \in \mathbb{R}^4$ and the input is the force $u_i = F_i \in [-100\,\text{N},100\,\text{N}]$ applied to the cart.
The carts are connected in a chain where each cart is coupled to its neighbors via a spring with stiffness $k = 0.1\,$N/m.
Let $M_{\mathrm{c}} = 2\,$kg and $m = 0.25\,$kg denote the masses of each cart and pendulum, respectively, let $l = 0.2\,$m denote the length of each pendulum, and denote the gravity of earth by $g = 9.81\,\text{m}/\text{s}^2$.
The continuous-time equations of motion of pendulum $i$ read
$\ddot{q}_i =  ( u_i + \frac{3}{4} m g \sin(\varphi_i) \cos(\varphi_i) - \frac{m l}{2} \dot \varphi_i^2 \sin(\varphi_i) + F_i^\mathrm{left} + F_i^\mathrm{right} ) / (M_{\mathrm{c}} + m - \frac{3}{4} m (\cos(\varphi_i))^2) $ and ${\ddot{\phi}_i = \frac{3g}{2l} \sin(\varphi_i) + \frac{3}{2l} \cos(\varphi_i) \ddot q_i}$.
If applicable, the coupling forces are $F_i^\mathrm{left} \doteq k (q_{i-1} - q_i)$ and $F_i^\mathrm{right} \doteq k (q_{i+1} - q_i)$.
We discretize the continuous-time dynamics using the explicit Runge-Kutta fourth-order method~(RK4) with a shooting interval ${h = 40}\,$ms and obtain the discrete-time model $f_i^h : \mathbb{R}^{4} \times \mathbb{R} \times \mathbb{R}^{n_i^\mathrm{in}} \rightarrow \mathbb{R}^4$.
For this discretization, we not only keep the control $u_i$ constant over the integration step, but also the neighboring positions $q_{i-1}$ and $q_{i+1}$ when evaluating the system dynamics.
This simplification preserves the coupling structure, i.e., the discrete-time dynamics of each subsystem only depend on the left and right neighbors.
However, the trade-off associated with this simplification is that the integration order with respect to the neighboring states reduces to one. 
A more accurate alternative is the distributed multiple shooting scheme~\cite{Savorgnan2011} which expresses the state trajectories as linear combinations of basis functions, e.g., Legendre polynomials, and matches the basis function coefficients between neighboring subsystems.

\begin{figure}
	\centering
	\includegraphics[width=0.8\columnwidth]{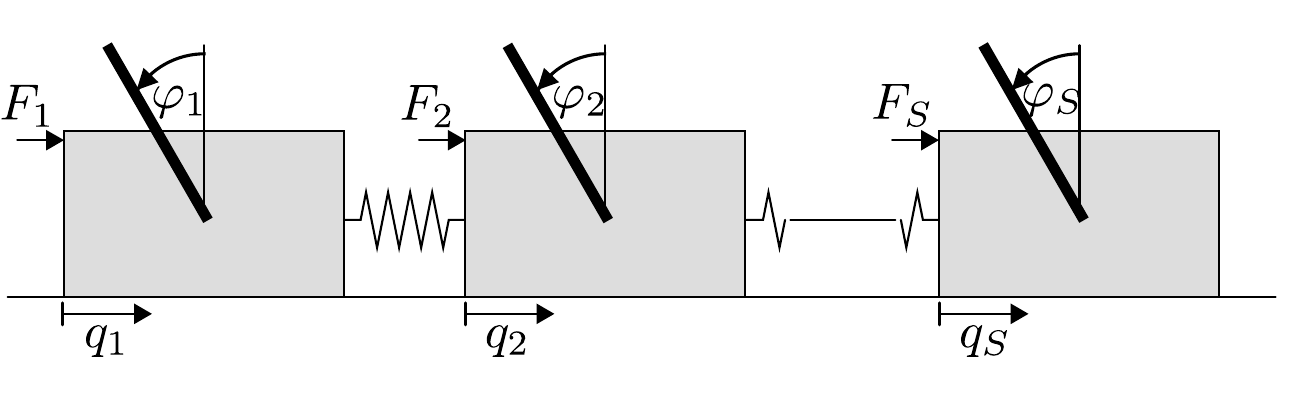}
	\caption{Coupled inverted pendulums on carts.}\label{fig:pendulum}
\end{figure}

\subsection{Optimal Control Problem Design}\label{sec:OCPdesign}

We design OCP~\eqref{docp} such that the value function is a Lyapunov function which satisfies inequalities~\eqref{eq:lyapReq}. 
For the sake of reducing the computational burden, our design does not enforce a terminal constraint in the OCP~\cite{Limon2006}. Thus, for the pendulum example, all inequality constraints are input constraints.
However, if desired, terminal constraints could also be considered. 
We choose separable quadratic stage costs $\ell_i(x_i,u_i) = x_i^\top Q_i x_i/2 +  u_i^\top R_i u_i/2,$ where we set $R_i = 0.001$ and take $Q_i = \text{diag}(1,10^{-4},10,10^{-4})$ from~\cite{Mills2009}.
We further design $V_{\text{f},i}(x_i) = {\beta_2}x_i^\top P_i x_i / 2$, where we find $\beta_2 \geq 1$ and $P_i \in \mathbb{R}^{4 \times 4}$ as follows.
First, we neglect coupling between neighbors, discretize the dynamics with RK4 at sampling interval $\delta = 40\,$ms, and linearize at the origin to obtain a model $(A_i,B_i)$ with $A_i \in \mathbb{R}^{4 \times 4}$ and $B_i \in \mathbb{R}^{4 \times 1}$ for all $i \in \mathcal{S}$. 
This allows to solve the algebraic Riccati equation \[
P_i = A_i^\top P_i A_i - (A_i^\top P_i B_i)(R_i + B_i^\top P_i B_i)^{-1} (B_i^\top P_i A_i) + Q_i
\] individually for each subsystem to obtain $P_i$ and and the terminal control law $u_i = K_i x_i$, where $K_i \doteq - (B_i^\top P_i B_i + R_i)^{-1} (B_i^\top P_i A_i)$.
Next, we discretize the centralized dynamics $f^\mathrm{c}$ including coupling with RK4 at $\delta = 40\,$ms and linearize at the origin to obtain the centralized system $(A,B)$ with $A \in \mathbb{R}^{4S \times 4S}$ and $B \in \mathbb{R}^{4S \times S}$.
For the example at hand, the terminal controller $K = \text{diag}(K_1,\dots,K_S)$ also stabilizes the coupled centralized system, i.e., the matrix $A_K = A + BK$ is Schur stable. 
Denote the centralized weight matrices as $Q \doteq \text{diag}(Q_1,\dots,Q_S)$, $R \doteq \text{diag}(R_1,\dots,R_S)$, and $P \doteq \text{diag}(P_1,\dots,P_S)$ and define  $Q_K \doteq Q + K^\top R K$.
To meet the sufficient decrease condition
\begin{equation*}
\Vf(f^\delta(x,Kx)) - \Vf(x) \leq - \ell(x,Kx),
\end{equation*} we increase $\beta_2$ until the matrix $\Delta Q \doteq \tilde{Q} - Q_K$ is positive definite, where $\tilde{Q} \doteq \beta_2 (P - A_K^\top P A_k)/\mu$ and where the paramter $\mu$ is in the notation of~\cite[Sec. 2.5.5]{Rawlings2019}.
We set $\mu = 1.01$ and obtain $\beta_2 = 1.1$.

\begin{rem}[Decentralized $V_\mathrm{f}$ design]
	The matrices $P_i$ and terminal controllers $K_i$ are computed by neglecting the coupling of the dynamics and by solving the Riccati equation for each subsystem individually.
	This simplification is not guaranteed to stabilize coupled linearized systems in general, but the design stabilizes the chain of pendulums considered here.
	Hence, the matrix $\Delta Q$ is positive definite for sufficiently large $\beta_2$.
	A similar design approach is suggested in~\cite[Remark 4]{Stewart2011}.
	More elaborate decentralized linear-quadratic control designs can be found in the classic textbook~\cite{Siljak1991}. \hfill $\square$
\end{rem} 

\subsection{Swing-up Simulation}

We consider the swing up of $S = 20$ pendulums to analyze the efficacy of the proposed RTI scheme.
Our Matlab implementation is available online and uses~CasADi to evaulate the derivatives in Step~\ref{stp:3} of dSQP~\cite{Andersson2019}.\footnote{\url{https://github.com/OptCon/real-time-dmpc}}
The quadratic subproblems in Step~\ref{admm-step:1} of ADMM are solved using~OSQP with tolerances $\epsilon_\text{abs} = \epsilon_\text{rel} = \epsilon_\text{prim} = \epsilon_\text{dual} = 10^{-8}$, where the tolerances are in the notation of~\cite{Stellato2020}.
After the control input is computed in each NMPC step, the centralized system is simulated using RK4 with integration step size equal to the control sampling interval $\delta$ to obtain the system state at the next sampling instant.
We select the penalty parameter $\rho$ of ADMM from the set $\{0.1, 1, 10, 100\}$ and choose $\rho = 1$.

We compare three test cases with varying initial conditions, OCP designs, and solver settings as summarized in Table~\ref{tab:testCases}. 
The pendulums initially rest in the lower equilibrium position $x_i(0) = (q_i(0),0,\pi,0)$, where the initial cart displacement is given in Table~\ref{tab:testCases}.
The goal is to steer all pendulums to the upright equilibrium position at $x_i = 0$.
The OCP is designed with quadratic weights as described in Subsection~\ref{sec:OCPdesign}.
The time horizon is chosen as $T = 0.4\,$s and the scaling factor in the objective is set to $\beta = 1$.
The further parameters in Table~\ref{tab:testCases} are the shooting interval $h$ in the OCP, the discrete-time horizon $N = T/h$, the chosen Hessian for QP~\eqref{eq:QP}, the maximum number of SQP iterations per NMPC step $k_\text{max}$, the maximum number of ADMM iterations per SQP iteration $l_\text{max}$, the total number of decision variables in the centralized OCP $n$, the total number of subsystem equality constraints $n_g$, the total number of subsystem inequality constraints $n_h$, and the number of consensus constraints $n_c$.
For all cases, the control sampling interval is set to $\delta = 40\,$ms.
For cases one and two, we choose the exact Hessian if $\nabla_{z_iz_i}^2 L_i$ is positive definite and otherwise we select the GN Hessian.
For case three, we always choose the GN Hessian approximation.
A quadratic penalty term with weight $10^{-5}$ is added to the objective~\eqref{docp:obj} for all state copies to meet Assumption~\ref{ass:kkt}~(ii).
We initialize dSQP in the first NMPC step at a solution found by IPOPT~\cite{Wachter2006}.
In all subsequent NMPC steps, we initialize dSQP with the solution produced in the previous control step.

Table~\ref{tab:sim_results} summarizes the simulation results.
We analyze the averaged closed-loop control performance
\begin{equation*}
J_\text{cl} \doteq \frac{1}{t_n+1}\sum_{t = 0}^{t_n} \sum_{i \in \mathcal{S}}  \ell_i(x_i(t),u_i(t)),
\end{equation*} where $t_n \doteq T_f/\delta$ and where $T_f = 10\,$s is the simulated time span.
Furthermore, we analyze the dSQP execution time per NMPC step on a desktop computer. 
We run each test case ten times to account for varying execution times in between runs and we take two different types of time measurements in each NMPC step.
In the first type, we measure the total execution time of one call to dSQP, which is summarized in the second and third columns from the left in Table~\ref{tab:sim_results}.
This includes running the specified number of SQP and ADMM iterations for all pendulums in series as well as the costly creation and destruction of intermediate data structures.
This is due to the prototypical nature of our Matlab implementation and would be avoided in embedded applications.
In the second type of time measurement, summarized in columns four to six of Table~\ref{tab:sim_results}, we measure only imperative code blocks that cannot be avoided in an efficient implementation, and we take measurements for each subsystem individually.
This includes calls to~CasADi for evaluating derivatives, calls to~OSQP for updating and solving QPs in ADMM, and computing Steps~\ref{admm-step2}--\ref{admm-step3} of ADMM.
Column six summarizes the percentage of per-subsystem solve times that were below the sampling interval.

The closed-loop system and optimizer trajectories are shown in Figures~\ref{fig:swingup_sys_1}--\ref{fig:swingup_sys_3}.
The top three plots in each figure display the cart positions, pendulum angles, and control inputs.
The optimizer evolutions in the bottom plots show the convergence of dSQP to the OCP solutions $p^\star$ found by~IPOPT.

In all test cases, the pendulums reach the upright equilibrium position while satisfying the input constraints.
The initial cart displacements in the first test case are less challenging than in the second case.
In fact, the increase in $k_\text{max}$ from the first to the second test case is necessary for a successful swing up due to the different initial conditions. 
That is, the optimizer settings of case one do not provide stability in simulation for case two.
While the per-subsystem computation time in the first test case is below the sampling interval $\delta$, the computation time of the second test case would not be real-time feasible.\footnote{Non-negligible computation times below the sampling interval can be compensated by solving the OCP for the subsequent control input~\cite{Findeisen2006}.}
Therefore, we increase the shooting interval in the third test case to reduce the OCP size and we select the GN Hessian to reduce the time spent evaluating derivatives such that case three requires less iterations per NMPC step for a successful swing up.
On the other hand, this reduces performance as can be seen from the prolonged settling time and the increase in~$J_\text{cl}$.

\begin{rem}[Stability, shooting interval, and GN Hessian]
Test case three selects a larger shooting interval and the GN Hessian to accelerate the computations. 
In general, choosing a shooting interval $h > \delta$ is a well-known technique for reducing RTI computation times~\cite{Diehl2002} and we found the GN Hessian to work well in hardware experiments~\cite{Stomberg2023,Stomberg2025a}.
However, while we observe stability for the pendulum simulation, the guarantees provided by Theorem~\ref{thm:drti} only hold if $h = \delta$ and if the exact Hessian $H \doteq \nabla_{zz}L$ is chosen.\hfill $\square$
\end{rem}

The time measurements, obtained for a prototypical Matlab implementation, show that the per-subsystem computation times are mostly below the control sampling interval.
This does not imply that an efficient decentralized implementation would necessarily be real-time feasible, in particular if communication latencies add to the computation time.\footnote{Latency measurements for a decentralized dSQP implementation indicate that a worst-case latency of approximately $10\,$ms can be expected for running $k_\text{max}=1$ SQP and $l_\text{max}=6$ ADMM iterations in a laboratory setting~\cite{Stomberg2023}.}
However, the obtained computation times indicate that control sampling intervals in the millisecond range are conceivable via the decentralized RTI scheme.

\begin{table}\caption{Test case specifications. All cases consider a sampling interval $\delta = 40\,$ms.}\label{tab:testCases}
	\centering
	\scalebox{0.95}{
		\begin{tabular}{ c c c c c c c c c c c }
			Case & $q_i(0)$ & $h\,$[ms] & $N$ & $H$ & $k_\text{max}$ & $l_\text{max}$ & $n$ & $n_g$ & $n_h$ & $n_c$\\
			\hline
			1 & $-1^{i}$ &  40 & 10 & exact & 1  & 6 & 1518 & 880 & 440 & 418\\
			2 & $i$ & 40 & 10 & exact & 3 & 6 & 1518 & 880 & 440 & 418\\
			3 & $i$ & 57 & 7 & GN & 2 & 3 & 1104 & 640 & 320  & 304
		\end{tabular}
	}
\end{table}

\begin{table}\caption{Computation times per NMPC step and closed-loop performance $J_\text{cl}$ for a prototypical Matlab implementation. The fast computation times in test cases one and three demonstrate the real-time feasibility of decentralized real-time iterations.}\label{tab:sim_results}
	\centering
	\scalebox{0.95}{
		\begin{tabular}{ c c c c c c c }
			Case & \multicolumn{2}{c}{all subsystems combined} &  \multicolumn{3}{c} {per subsystem} & $J_\text{cl}$\\
			& median [ms] & max. [ms] & median [ms] & max. [ms] & $\leq \delta$ & \\
			\hline
			1 & 239.07 & 448.19 & 12.03 & 29.24 & 100\,\% & 65.86\\
			2 & 697.04 & 2394.02 & 35.10 & 210.38 & 93.54\,\% & 156.05\\
			3 & 227.06 & 746.51  & 10.76 & 60.25 & 99.90\,\% & 180.66
		\end{tabular}
	}
\end{table}

\begin{figure}
	\includegraphics[width=\columnwidth]{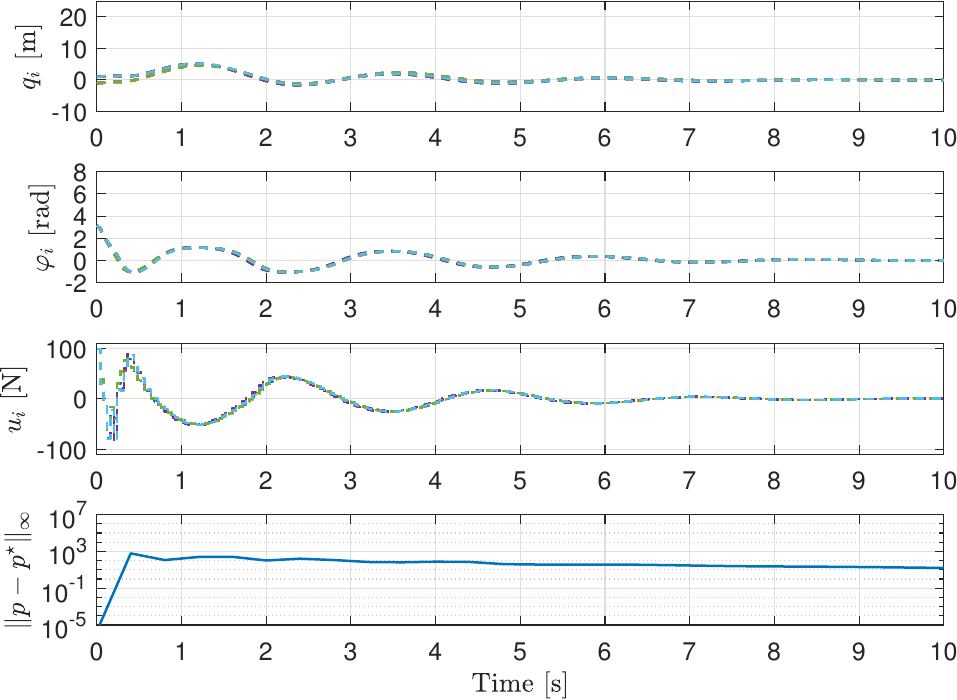}
	\caption{Test case 1: closed-loop system and optimizer convergence with 6 ADMM iterations per NMPC step.}\label{fig:swingup_sys_1}
\end{figure}

\begin{figure}
	\includegraphics[width=\columnwidth]{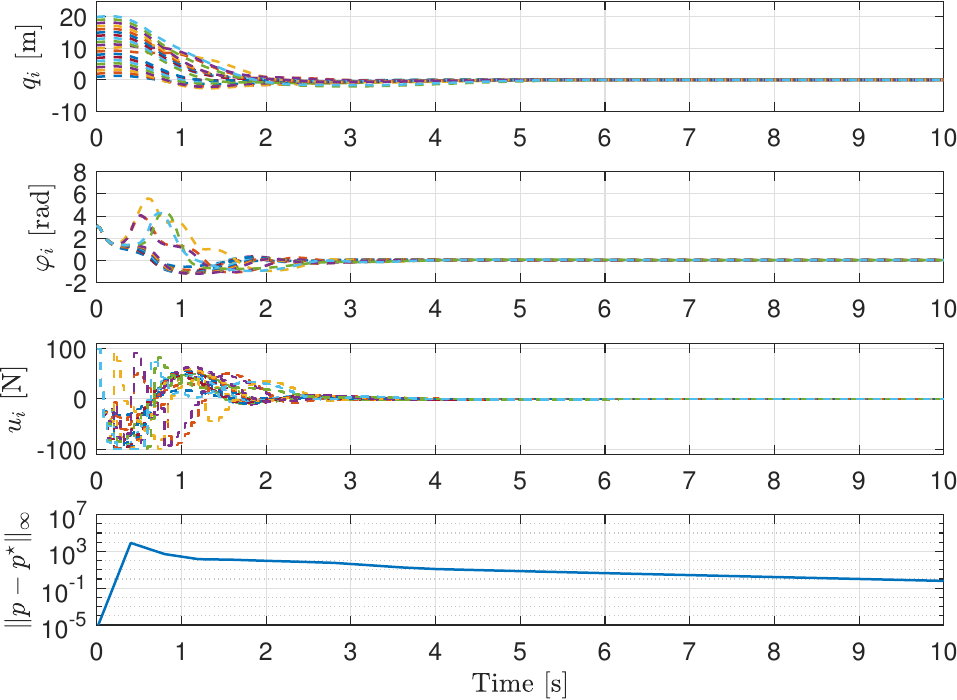}
	\caption{Test case 2: closed-loop system and optimizer convergence for a challenging initial condition with 18 ADMM iterations per NMPC step.}\label{fig:swingup_sys_2}
\end{figure}

\begin{figure}
	\includegraphics[width=\columnwidth]{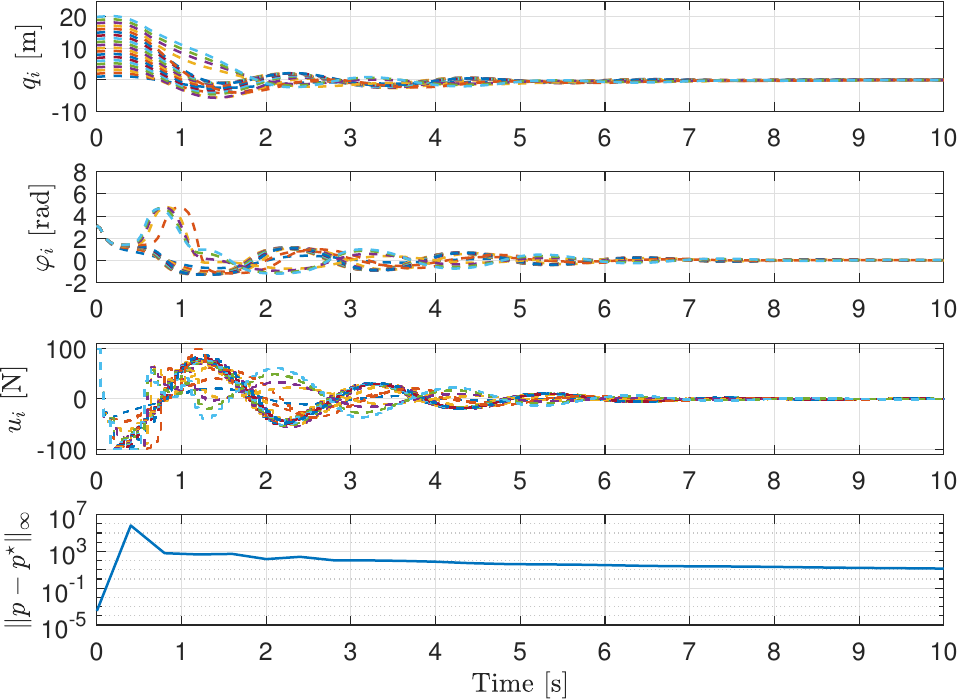}
	\caption{Test case 3: closed-loop system and optimizer convergence with 6 ADMM iterations per NMPC step.}
	\label{fig:swingup_sys_3}
\end{figure}

\subsection{Parameter Estimation and Validity of Assumptions}

The OCP design discussed in subsection~\ref{sec:OCPdesign} meets the value function requirements~\eqref{eq:lyapReq} in Assumption~\ref{ass:Lyapunov} and the pendulum dynamics satisfy Assumption~\ref{ass:Lipschitz_sys}.
Numerical a-posteriori analyses further show that the KKT points found by~IPOPT satisfy the regularity Assumption~\ref{ass:kkt} if the system is close to the setpoint.
The continuity conditions in Assumptions~\ref{ass:Lyapunov} and~\ref{ass:Lipschitz_z} thus hold for the identified local minima. 
Finally, Assumption~\ref{ass:c3} is satisfied, because the pendulum dynamics are sufficiently smooth.

The bounds~\eqref{eq:lmax} on the ADMM iterations $l_\mathrm{max}$ and~\eqref{eq:RTIsuff} on $\bar{\delta}$ offer a compromise between optimiality, i.e., large $l_\mathrm{max}$, and sampling frequency, i.e., small $\bar{\delta}$.
To obtain stability for a large sampling interval $\bar{\delta}$ via~\eqref{eq:RTIsuff}, the outer contraction factor $a_p$ must be small. 
By Theorem~\ref{thm:dsqp}, this requires $a$ to be small, which results in more ADMM iterations $l_\mathrm{max}$ via~\eqref{eq:lmax}.
That is, for low sampling frequencies, ADMM must solve the SQP subproblems to greater accuracy.
To estimate the number of ADMM iterations $l_\mathrm{max}$ sufficient to guarantee stability for the sampling interal $\delta = 40\,$ms chosen in the simulations, we proceed as follows:
We estimate the constants for calculating $\delta_5$ in Appendix~\ref{sec:app-rti} via simulations close to the setpoint with the ideal centralized NMPC feedback law~$\kappa_\mathrm{c}(x)$ using IPOPT, similarly to~\cite{Zanelli2021}.
These simulations yield constants $a_1 = 0.5326$, $a_2 = 7.1530$, $a_3 = 0.2113$, $L_{f,x}^\mathrm{c} = 86.2704$, $L_{f,u}^\mathrm{c} = 3.6685$, $L_{p,x} = 51.3647$, $L_{V,x} = 1.7908$, and $L_{V,p} = 207.1090$.
Then, the QP approximation of the OCP at the setpoint yields $c_1 = 1.7321$ and $c_2 = 251.5737$.
Finally, we sample the ADMM LTI dynamics~\eqref{eq:admmlti} in simulations for random $w$ and obtain $a_w = 0.9989$.
As a result, $l_\mathrm{max} = 24007$ ADMM iterations are sufficient to guarantee stability for~$\delta_5 = 40\,$ms.
Compared with the settings in Table~\ref{tab:testCases}, this estimate is quite conservative.
Here, the conservatism primarily stems from the Lipschitz constants in the system dynamics and the contraction factor~$a_w$.
The conservatism could thus be reduced by more accurate estimates for these constants, for instance through tighter ADMM convergence guarantees or elaborate simulation.
In hardware experiments and simulations, we observe good control performance already for 2--30 ADMM iterations per MPC step, depending on the application~\cite{Stomberg2023,Stomberg2025a,Stomberg2025b}.

\section{Conclusion}

This paper has presented a novel decentralized RTI scheme for distributed NMPC based on dSQP.
The proposed scheme applies finitely many optimizer iterations per control step and does not require subsystems to exchange information with a coordinator.
Stability guarantees are proven for the system-optimizer dynamics in closed loop by combining centralized RTI stability guarantees with novel dSQP convergence results.
Numerical simulations demonstrate the efficacy of the proposed scheme for a chain of coupled inverted pendulums.
Future work will consider further mechatronic systems and hardware experiments.

\appendices
\section{Centralized RTI Stability}\label{sec:app-rti}

This section summarizes the steps for computing the sufficent sampling interval $\bar{\delta}>0$ and optimizer initialization radius $\tilde{r}_p > 0$ which guarantee stability of centralized RTIs~\cite{Zanelli2021}.
\begin{lem}[Lipschitz discrete-time dynamics~\cite{Zanelli2021}]\label{lem:lipschitzd}
	Let Assumption~\ref{ass:Lipschitz_sys} hold.
	Then, there exists a positive finite constant $\delta_1$ such that, if $x \in \mathbb{X}_{\bar{V}}$, $ p \in \mathcal{B}(\bar{p}(x),r_p')$, and $\delta \leq \delta_1$, then
	\begin{equation*}
	\left \| f^\delta(x,M_{u,p}p) - x \right \| \leq \delta \cdot \left( L_{f,x}^{\delta_1} \|x\| + L_{f,u}^{\delta_1} \left\| M_{u,p} p \right\| \right),
	\end{equation*} where $L_{f,x}^{\delta_1} \doteq e^{L_{f,x}^\mathrm{c} \delta_1} L_{f,x}^\mathrm{c}$ and  $L_{f,u}^{\delta_1} \doteq e^{L_{f,x}^\mathrm{c} \delta_1} L_{f,u}^\mathrm{c}$.
	Moreover, if $\delta \leq \delta_1$, then
	\begin{equation} \label{eq:xplusbound}
	\left\| f^\delta(x,u') - f^\delta(x,u) \right\| \leq \delta L_{f,u}^{\delta_1} \left\|u' - u\right\|
	\end{equation}
	for all $x \in \mathbb{X}_{\bar{V}}$, all $u' = M_{u,p}p'$, $u = M_{u,p}p$ such that $p,p' \in \mathcal{B}(\bar{p}(x),r_p')$.	
	\hfill $\square$
\end{lem}

\begin{lem}[Forward invariance under perturbed optimizer]\label{lem:fwdx}
	Let Assumptions~\ref{ass:Lyapunov} and~\ref{ass:Lipschitz_sys} hold.
	Then, there exists a positive constant $r_p''\leq r_p'$ such that $f^\delta(x,M_{u,p}p) \in \mathbb{X}_{\bar{V}}$ for all $x \in \mathbb{X}_{\bar{V}}$, all $p \in \mathcal{B}(\bar{p}(x),r_p'')$, and if $\delta \leq \delta_1$. \hfill $\square$
\end{lem}
\begin{proof}
	We define the shorthands $x_+^\star \doteq f^\delta(x,M_{u,p}\bar{p}(x))$ and $x_+ \doteq f^\delta(x,M_{u,p}{p})$.
	Let $r_p'' \leq r_p'$ with $r_p'$ from Assumption~\ref{ass:Lipschitz_sys}.
	By the Lyapunov decrease condition in Assumption~\ref{ass:Lyapunov}, we have that
	$V(x_+^\star) \leq V(x) - \delta a_3 \|x\|^2$.
	If $x = 0$, then $V(x) = 0$ and thus $V(x_+^\star) = 0 < \bar{V}$.
	If $x \neq 0$, then $V(x_+^\star) \leq \bar{V} - \delta a_3 \|x\|^2 < \bar{V}$.
	Thus, there exists a constant $\epsilon_V > 0$ such that $V(x_+^\star) + \epsilon_V \leq \bar{V}$.
	Recall that, by Assumption~\ref{ass:Lyapunov}, $\mathbb{X}_{\bar{V}}$ lies in the interior of $\mathbb{X}_0$ and that $V$ is continuous over $\mathbb{X}_0$.
	Thus, there exists a constant $\epsilon_x > 0$ such that $| V(x') - V(x_+^\star) | < \epsilon_V$ for all $\|x' - x_+^\star\| < \varepsilon_x$~\cite[App. A.11]{Rawlings2019}.
	By inserting $\| p - \bar{p}(x) \| \leq r_p''$ into~\eqref{eq:xplusbound}, we have that $\| x_+ - x_+^\star \| \leq \delta L_{f,u}^{\delta_1} r_p''$.
	For $r_p'' > 0$ sufficiently small, $\|x_+ - x_+^\star \| < \epsilon_x$ and so $V(x_+) < V(x_+^\star) + \epsilon_V \leq \bar{V}$.
	By the definition of the level set $\mathbb{X}_{\bar{V}}$, we obtain that $x_+ \in \mathbb{X}_{\bar{V}}$, which concludes the proof.
\end{proof}

Define the shorthands $x_t \doteq x(t)$, $p_{t} \doteq p^{k_\mathrm{max}}(t)$, and $p^\star_t \doteq \bar{p}(x(t))$, and similarly for $\cdot_{t+1}$.
We next adapt~\cite[Lem. 11]{Zanelli2021} to the q-linear optimizer convergence Assumption~\ref{ass:qlin}.
\begin{lem}[Contraction]\label{lem:optimizer-contraction}
	Let Assumptions~\ref{ass:Lyapunov}--\ref{ass:Lipschitz_sys} hold and consider the system-optimizer dynamics~\eqref{eq:sysoptdyn}.
	Then, there exist positive constants $r_p \leq \min\{\hat{r}_p,r_p''\}$ and $r_x 
	\leq \hat{r}_x$ such that, for all $x_t \in \mathbb{X}_{\bar{V}}$ and all $\|p_t - p^\star_t\| \leq r_p$, the following holds.
	If $\delta \leq \delta_1$ and if $\|x_{t+1} - x_t\| \leq r_x$, then
	\begin{align}\label{eq:optimizer-contraction}
	\| p_{t+1} - p^\star_{t+1} \| \leq a_p \| p_t - p^\star_t \| + a_p L_{p,x}  \| x_{t+1} - x_t \|.
	\end{align}~\hfill$\square$
\end{lem}

\begin{proof}
		Since $r_p \leq r_p''$, Lemma~\ref{lem:fwdx} yields that $x_{t+1} = f^\delta(x_t,M_{u,p}p_t) \in \mathbb{X}_{\bar{V}}$.
		Assumption~\ref{ass:Lipschitz_z} on the Lipschitz continuity of the KKT point thus states that
		$$
		\| p^\star_{t+1} - p^\star_t\| \leq L_{p,x} \|x_{t+1} - x_t \| \leq L_{p,x}r_x.
		$$ 
		By further setting $r_p \leq \min\{r_p'',\hat{r}_p - L_{p,x}r_x\}$ with ${r_x < \min\{\hat{r}_x,\hat{r}_p/L_{p,x}\}}$, we obtain
		\begin{align*}
		\|p_t - p^\star_{t+1}\| &\leq \|p_t - p^\star_{t}\| + \| p^\star_{t+1} - p^\star_t\|\\
		&\leq r_p + L_{p,x}r_x \leq \hat{r}_p.
		\end{align*}
		Thus, $\|p_t - p^\star_{t+1}\|$ lies inside the q-linear convergence region of the optimizer.
		Assumptions~\ref{ass:Lipschitz_z} and~\ref{ass:qlin} hence yield
		\begin{align*}
		\| p_{t+1} - p^\star_{t+1} \| &\leq  (a_p)^{k_\mathrm{max}} \|p_t - p_{t+1}^\star\| \\
		&\leq a_p \|p_t - p_{t+1}^\star\| \\
		&\leq a_p \|p_t - p_{t}^\star\| + a_p \| p^\star_{t+1} - p^\star_t\|\\
		&\leq a_p \|p_t - p_{t}^\star\| + a_p L_{p,x} \| x_{t+1} - x_t\|,
		\end{align*} where we have used that $(a_p)^{k_\mathrm{max}} \leq a_p < 1$. This finishes the proof.
\end{proof}

Furthermore, define the constants $\eta \doteq L_{f,x}^{\delta_1} + L_{f,u}^{\delta_1} L_{p,x}$, $r_{\bar{V}} \doteq (\bar{V}/a_1)^{1/2}$, 
	\begin{equation*}
	\delta_3 \doteq \min \left\{\delta,\delta_1,\frac{r_x}{\eta r_{\bar{V}} + L_{f,u}^{\delta_1} r_p } , \frac{ r_p (1- a_p ) }{ L_{p,x} a_p \left( L_{f,u}^{\delta_1} r_p + \eta r_{\bar{V}} \right)  } \right\},
	\end{equation*} 
	$\kappa \hspace*{-0.5mm}\doteq\hspace*{-0.5mm} a_p \left( 1\hspace*{-0.5mm}+\hspace*{-0.5mm} \delta_3 L_{p,x} L_{f,u}^{\delta_1} \right)$,
	$L_V \doteq 2 \bar{V}^{1/2} L_{V,x}$,
	${\bar{a} \doteq a_3/a_2}$,
	$L_e \doteq L_V L_{f,u}^{\delta_1}$,
	$L_{V,p} \doteq L_{f,u}^\mathrm{c} e^{\delta_1 L_{f,x}^\mathrm{c}} L_{V,x}$, and
	${\beta' \doteq ( \bar{a} \sqrt{a_1} )/(4 L_{p,x} a_p \eta)}$.

The sampling interval $\bar{\delta}$ and optimizer initialization radius $\tilde{r}_p$ sufficient for closed-loop stability read~\cite{Zanelli2021}
\begin{equation}\label{eq:RTIsuff}
\bar{\delta} \doteq \min\{\delta_3,\delta_4',\delta_5\} \quad \text{and} \quad \tilde{r}_p \doteq \min\left\{ r_p, \bar{a} \bar{V} / L_e \right\},
\end{equation} where
$$
\delta_4' \doteq \frac{(1 - \kappa) \tilde{r}_p \sqrt{a_1}}{\bar{V}^{1/2} L_{p,x} a_p \eta}, \quad \text{and} \quad \delta_5 \doteq \frac{ \beta' (1-\kappa) }{L_{V,p}}.
$$

\section{NLP Formulation and Decentralized ADMM}\label{app:example}

This section presents an example to demonstrate the reformulation of OCP~\eqref{ocp} as a partially separable NLP~\eqref{eq:sepForm} and the decentralization of the ADMM averaging step.
\begin{example}[OCP as partially separable NLP~\cite{Stomberg2022}]\label{ex:coupling}
	Consider a set $\mathcal{S} = \{1,2\}$ of subsystems with $x_i,u_i\in \mathbb{R}$ for all $i \in \mathcal{S}$ governed by the dynamics
	\begin{align*}
	x_1(t+1) &= x_1(t) + u_1(t), \quad x_1(0) = x_{1,0}\\
	x_2(t+1) &= x_1(t) + x_2(t), \quad x_2(0) = x_{2,0}.
	\end{align*} 
	For a horizon $N = 1$, the decision variables in OCP~\eqref{ocp} then read $\bar{\boldsymbol{x}} \doteq (\bar{x}_1(0),\bar{x}_1(1) , \bar{x}_2(0),\bar{x}_2(1))$ and $\bar{\boldsymbol{u}} \doteq \bar{u}_1(0)$.
	To obtain a partially separable NLP, we define the state copy $\bar{v}_2 \doteq \bar{x}_1$ and the decision variables
	$z_1 \doteq (\bar{x}_1(0),\bar{x}_1(1),\bar{u}_1(0))$ and $z_2 \doteq (\bar{x}_2(0),\bar{x}_2(1),\bar{v}_2(0))$.
	The dynamics can be reformulated as subsystem constraints~\eqref{eq:sepProbGi},
	\begin{align*}
	g_1(z_1) &\doteq \begin{bmatrix} \bar{x}_1(0) + \bar{u}_1(0) - \bar{x}_1(1) \\ \bar{x}_1(0) - x_{1,0}\end{bmatrix}\phantom{.} \\ 
	g_2(z_2) &\doteq \begin{bmatrix} \bar{x}_2(0) + \bar{v}_2(0) - \bar{x}_2(1) \\ \bar{x}_2(0) - x_{2,0} \end{bmatrix}.
	\end{align*} The constraints~\eqref{eq:consConstr} couple original and copied states,
	\begin{align*}
		\underbrace{\begin{bmatrix} 1 & 0 & 0 \end{bmatrix}}_{\doteq E_1} z_1 + \underbrace{\begin{bmatrix} 0 & 0 & -1 \end{bmatrix}}_{\doteq E_2} z_2  = 0.
	\end{align*} 
	\end{example}\vspace*{-0.2cm}~\hfill~$\square$
	
	Example~\ref{ex:coupling} reformulates coupled dynamics. Coupled costs~\eqref{docp:obj} and constraints~\eqref{docp:xixj} can be treated similarly.	
	The reformulation allows to decentralize ADMM as follows. 
	The sparsity in $E_1$ and $E_2$ together with $c = 0$ allow to view NLP~\eqref{eq:sepForm} as a so-called consensus problem, where each subsystem optimizes over a selection of the centralized decision variables $(\bar{\boldsymbol{x}},\bar{\boldsymbol{u}})$.
	Specifically for Example~\ref{ex:coupling}, Subsystem~1 optimizes over $(\bar{x}_1(0),\bar{x}_1(1),\bar{u}_1(0))$ and Subsystem~2 optimizes over $(\bar{x}_1(0),\bar{x}_2(0),\bar{x}_2(1))$.
	A derivation of the decentralized implementation of Step~\eqref{eq:admm2} for general consensus problems is provided in~\cite[Ch. 7.2]{Boyd2011}.
	To make a closer connection to NLP~\eqref{eq:sepForm}, recall from the proof of Lemma~\ref{lem:admmConv} that the $z$-update in ADMM is given by
	$
	z^{l+1} = M_{\mathrm{avg}} (y^l + \gamma^l/\rho).
	$
	For Example~\ref{ex:coupling}, the ADMM averaging matrix $M_\mathrm{avg}$ reads
	$$
	M_\mathrm{avg} = I - E^\top (E E^\top)^{-1} E = \left[
	\begin{smallmatrix} 0.5 & 0 & 0 & 0 & 0 & 0.5 \\
	 0 & 1 & 0 & 0 & 0 & 0 \\
	 0 & 0 & 1 & 0 & 0 & 0\\
	 0 & 0 & 0 & 1 & 0 & 0\\
	 0 & 0 & 0 & 0 & 1 & 0\\
	 0.5 & 0 & 0 & 0 & 0 & 0.5  \end{smallmatrix} \right].
	$$
	That is, the update for the coupled variable $\bar{x}_1(0)$ is
	\begin{equation*}\label{eq:ex-avg}
	[z_1]_1^{l+1} = [z_2]_3^{l+1} = \frac{[y_1]_1^{l+1}  +  [y_2]_3^{l+1}  + ([\gamma_1]_1^l + [\gamma_2]_3^l) /\rho}{2}.
	\end{equation*}
	A common approach to decentralize ADMM in DMPC is to send the variables on the right hand side of the above average to the subsystem with the original state, i.e., Subsystem~1 in Example~\ref{ex:coupling}.
	Then, the average is evaluated and the result is sent to the subsystems with the copied states, i.e., Subsystem~2 in Example~\ref{ex:coupling}. For more details, see~\cite{Bestler2019,Stomberg2023,Stomberg2025a}.

\renewcommand*{\bibfont}{\footnotesize}
\printbibliography

\begin{IEEEbiography}
	[{\includegraphics[width=1in,height=1.25in,clip,keepaspectratio]{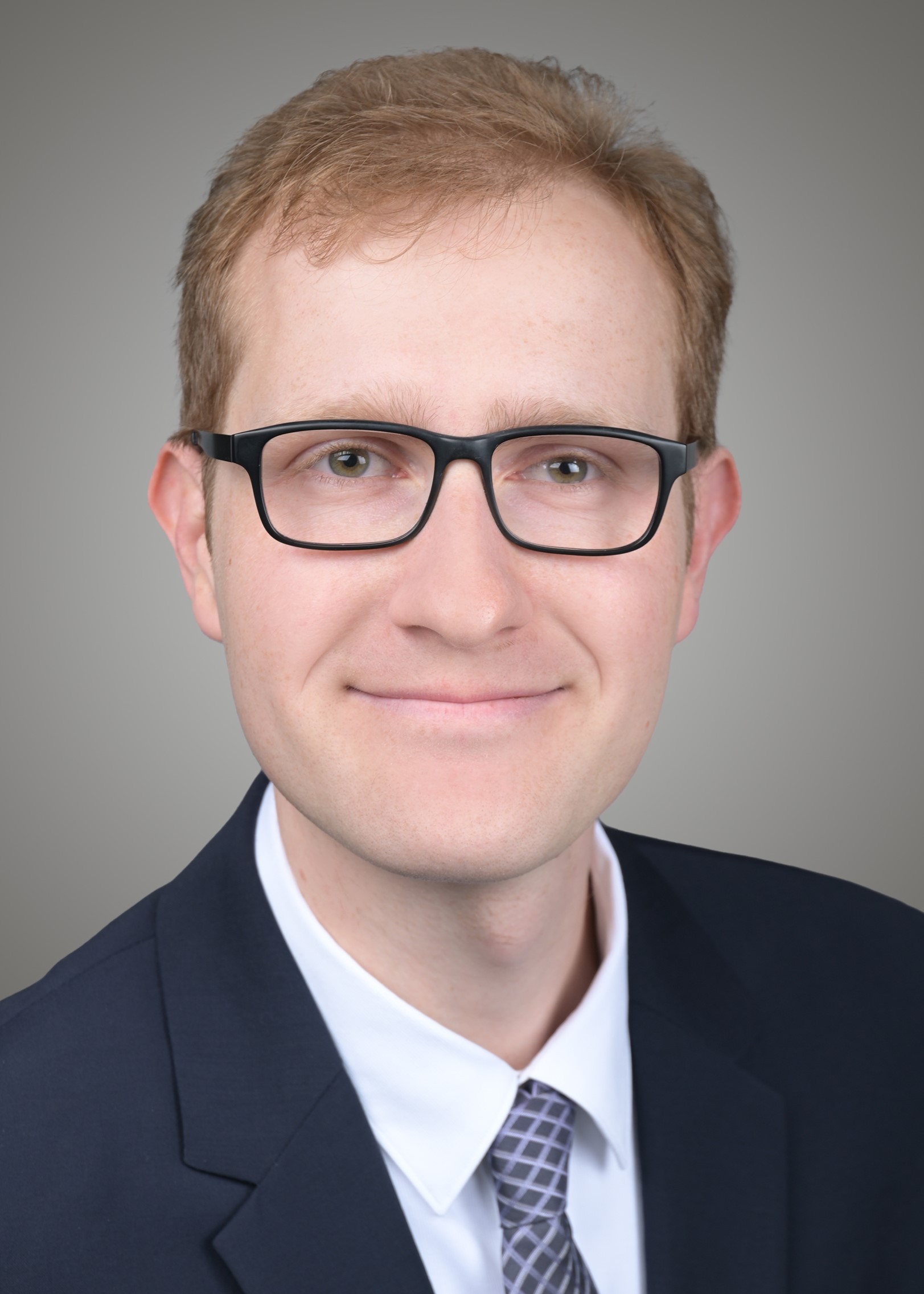}}]{Gösta Stomberg} (Member, IEEE) received the M.Sc. degree in mechanical and computational engineering from the Technical University of Darmstadt, Germany, in 2019. 
	Since 2020, he has been a PhD Student, first at TU Dortmund University, Germany, and, since 2024, at Hamburg University of Technology, Germany.
	His research interests include distributed optimization and model predictive control.
\end{IEEEbiography}

\begin{IEEEbiography}
	[{\includegraphics[width=1in,height=1.25in,clip,keepaspectratio]{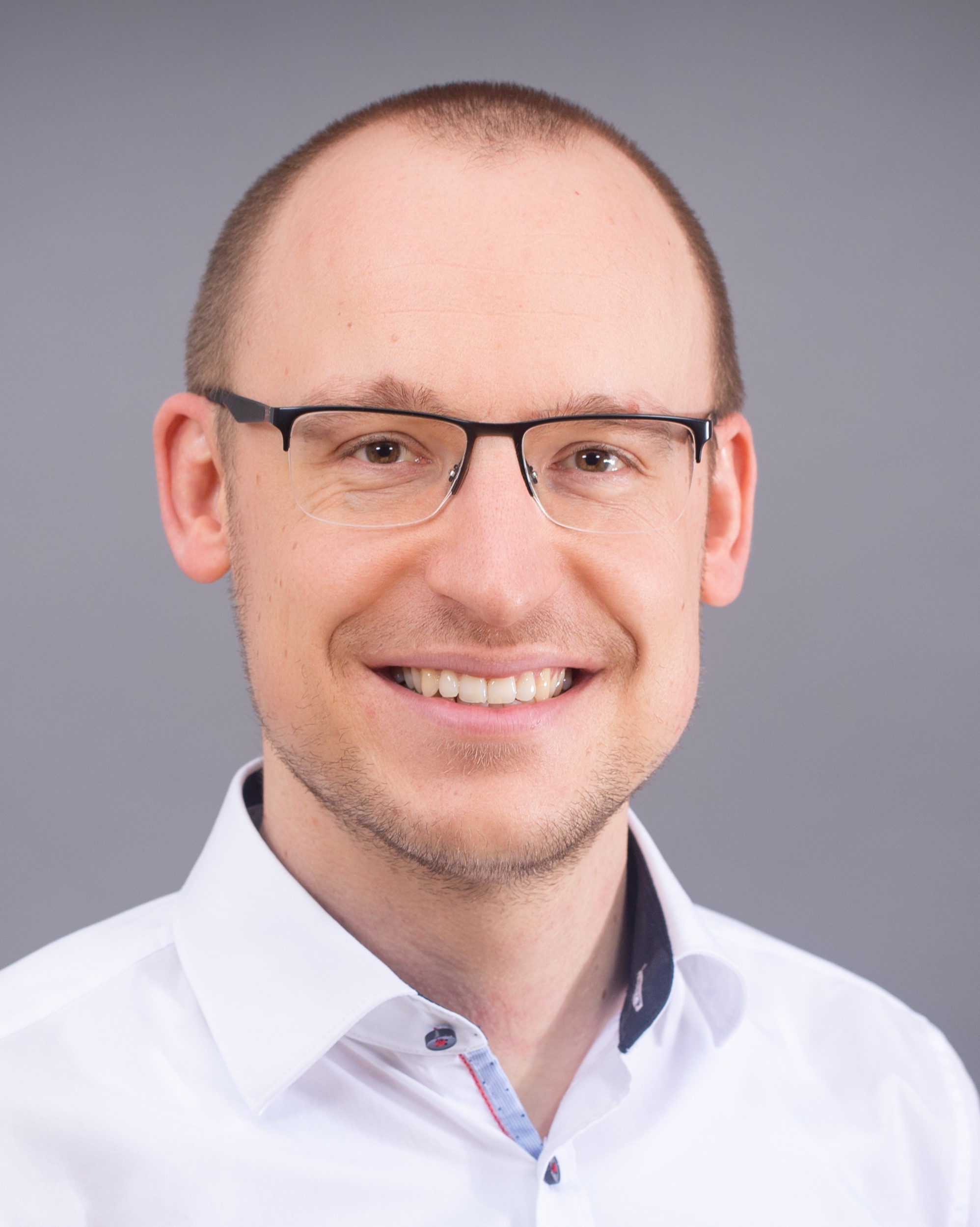}}]{Alexander Engelmann} (Member, IEEE) received the M.Sc. degree in electrical engineering and information technology and the Ph.D. degree in informatics from the Karlsruhe Institute of Technology, Germany, in 2016 and 2020, respectively. From 2020 to 2024, he was a postdoctoral researcher at the Institute of Energy Systems, Energy Efficiency and Energy Economics at TU Dortmund University, Germany. In 2024, he joined logarithmo GmbH, Dortmund, Germany, where he focuses on large scale optimization for power system operation in an industrial context. 
\end{IEEEbiography}

\begin{IEEEbiography}
	[{\includegraphics[width=1in,height=1.25in,clip,keepaspectratio]{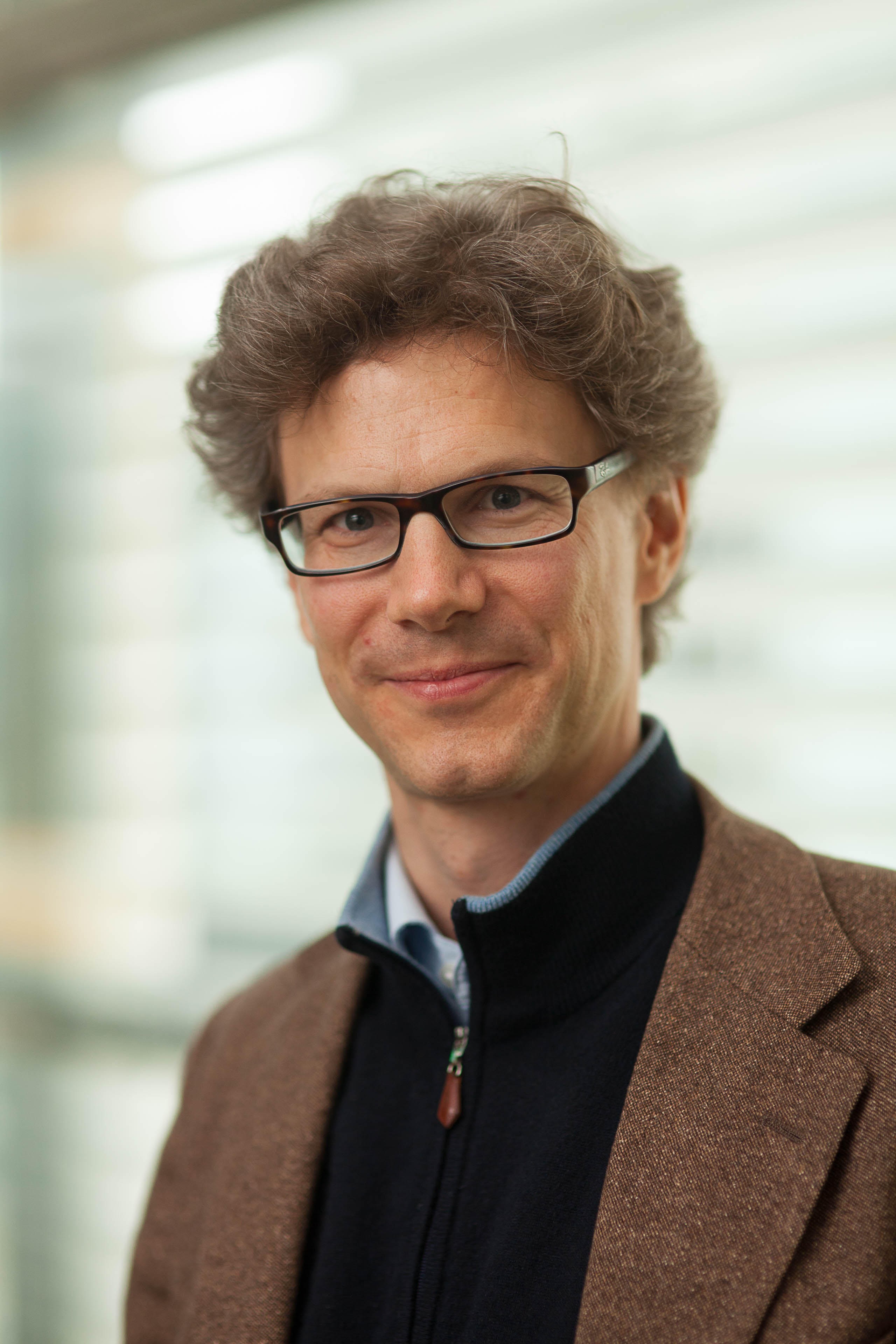}}]{Moritz Diehl} studied physics and mathematics at Heidelberg, Germany, and Cambridge, UK, from 1993-1999, and received his Ph.D. degree in Scientific Computing from Heidelberg University in 2001. From 2006 to 2013, he was a professor at KU Leuven University, Belgium, and served as the Principal Investigator of KU Leuven's Optimization in Engineering Center OPTEC.  Since 2013, he is full professor at the University of Freiburg, Germany, where he heads the  Systems Control and Optimization Laboratory, in the Department of Microsystems Engineering (IMTEK), and is also affiliated to the Department of Mathematics. Since 2023, he serves as managing director of Freiburg University’s Center for Renewable Energy (ZEE). His research interests are in optimization and control, spanning from numerical method and software development to applications in different branches of engineering, with a focus on embedded systems and on renewable energy systems.	
\end{IEEEbiography}

\begin{IEEEbiography}
	[{\includegraphics[width=1in,height=1.25in,clip,keepaspectratio]{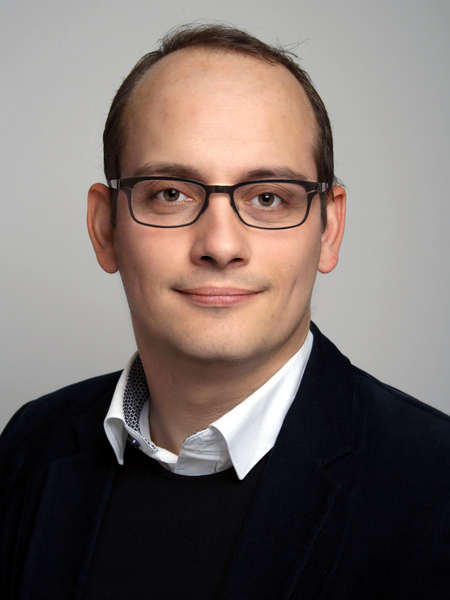}}]{Timm Faulwasser} (Senior Member, IEEE) is a full professor in the School of Electrical Engineering, Computer Science and Mathematics at Hamburg University of Technology, while before he held a professorship at TU Dortmund University. He has studied Engineering Cybernetics with minor in philosophy at the University of Stuttgart (2000-2006). After doctoral studies in the International Max Planck Research School for Analysis, Design and Optimization in Chemical and Biochemical Process Engineering Magdeburg he obtained his PhD from the Department of Electrical Engineering and Information Technology at Otto-von-Guericke-University Magdeburg, Germany in 2012. He has been postdoctoral researcher at École Polytechnique Fédérale de Lausanne (2013-2016) and senior researcher at Karlsruhe Institute of Technology (2015-2019). Previously, Timm was a member of the IEEE-CSS Conference Editorial Board and associate editor of the European Journal of Control. Currently, he serves as associate editor for the IEEE Transactions on Automatic Control, the IEEE Control System Letters, and Mathematics of Control Systems and Signals. He received the 2021-2023 Automatica Paper Prize and the European Control Award 2025. His current research interests are optimization-based and data-driven control of stochastic and nonlinear systems as well as systems and control approaches to learning. 
\end{IEEEbiography}

\end{document}